\documentclass[12pt, a4paper]{amsart}
\usepackage{amsmath}
\usepackage{geometry,amsthm,graphics,tabularx,amssymb,shapepar}
\usepackage[cmyk]{xcolor}
\usepackage{appendix}
\usepackage{amscd}
\usepackage[all]{xypic}
\usepackage{ulem}
\usepackage{bm}
\usepackage{extarrows}

\makeatletter
\newcommand*{\rom}[1]{\expandafter\@slowromancap\romannumeral #1@}
\makeatother

\newcommand{\BC}{{\mathbb {C}}}

\newcommand{\BN}{{\mathbb {N}}}

\newcommand{\BR}{{\mathbb {R}}}

\newcommand{\CI}{{\mathcal {I}}}

\newcommand{\CO}{{\mathcal {O}}}

\newcommand{\RC}{{\mathrm {C}}}

\newcommand{\Hom}{{\mathrm{Hom}}}

\newcommand{\Spec}{{\mathrm{Spec}}}

\newcommand{\wt}{\widetilde}
\newcommand{\wh}{\widehat}

\renewcommand{\r}{\mathfrak r}

\newcommand{\m}{\mathfrak m}

\newcommand{\C}{\mathbb{C}}
\newcommand{\R}{\mathbb R}

\newcommand{\M}{\mathbf{M}}

\newcommand{\J}{\mathbf{J}}

\newcommand{\cf}{\textit{cf}.~}
\newcommand{\la}{\langle}
\newcommand{\ra}{\rangle}

\newcommand{\be}{\begin {equation}}
\newcommand{\ee}{\end {equation}}
\newcommand{\bee}{\begin {equation*}}
\newcommand{\eee}{\end {equation*}}

\newcommand{\qaq}{\quad\textrm{and}\quad}

\renewcommand{\mid}{\,:\,}

\theoremstyle{Theorem}

\theoremstyle{Theorem}

\theoremstyle{Theorem}

\theoremstyle{Theorem}

\theoremstyle{Plain}

\theoremstyle{remark}

\theoremstyle{remark}

\theoremstyle{Definition}
\newtheorem{dfn}{Definition}[section]

\newtheorem{cord}[dfn]{Corollary}
\newtheorem{prpd}[dfn]{Proposition}
\newtheorem{thmd}[dfn]{Theorem}
\newtheorem{lemd}[dfn]{Lemma}
\newtheorem{remarkd}[dfn]{Remark}
\newtheorem{exampled}[dfn]{Example}
\numberwithin{equation}{section}

\begin{document}

\title[Formal manifolds]{Formal manifolds: local structure of morphisms, and formal submanifolds}

\author[F. Chen]{Fulin Chen}
\address{School of Mathematical Sciences, Xiamen University,
 Xiamen, 361005, China} \email{chenf@xmu.edu.cn}

\author[B. Sun]{Binyong Sun}
\address{Institute for Advanced Study in Mathematics \& New Cornerstone Science Laboratory, Zhejiang University,  Hangzhou, 310058, China}
\email{sunbinyong@zju.edu.cn}

\author[C. Wang]{Chuyun Wang}
\address{Institute for Theoretical
 Sciences and School of Science, Westlake University,
 Hangzhou, 310030, China}
\email{wangchuyun@westlake.edu.cn}


\begin{abstract}
This is a paper in a series that studies smooth relative Lie
algebra homologies and cohomologies based on the theory of formal manifolds
and formal Lie groups. In three previous papers, we introduce the notion of formal manifolds and study their basic theory, focusing on function spaces and Poincar\'e’s lemma.
 In this paper, we further explore the foundational framework of formal manifolds, including the local structure of constant rank morphisms (such as inverse function theorem and constant rank theorems) as well as the theory of formal submanifolds.
 \end{abstract}

\maketitle

 \tableofcontents

\section{Introduction and  main results}
This paper is a continuation of \cite{CSW1}, \cite{CSW2} and \cite{CSW3}.
We refer to the Introduction of \cite{CSW1} for the background and motivations for the study of formal manifolds.  

\subsection{Basics on formal manifolds}
We first recall from \cite{CSW1} some notions concerning formal manifolds.
Let $N$ be a smooth manifold, and let  $k\in \BN$. 
For every open subset $U\subset N$, write $\RC^\infty(U)$ for the $\BC$-algebra of complex-valued smooth functions on $U$, and write 
	\[
	\CO_N^{(k)}(U):=\RC^\infty(U)[[y_1, y_2, \dots, y_k]]
	\]
 for the $\BC$-algebra of formal power series with coefficients in $\RC^\infty(U)$. With the obvious restriction maps, we have a sheaf 
 \[\CO_N^{(k)}:\ U\mapsto \CO_N^{(k)}(U)\] of $\BC$-algebras. This gives a locally ringed space 
 \[N^{(k)}:=(N, \CO_N^{(k)})\]  over $\mathrm{Spec}(\BC)$.

\begin{dfn}\label{def:formalmanifold}
	A formal manifold  is a locally ringed space $(M, \CO)$ over $\mathrm{Spec}(\BC)$ such that
	\begin{itemize}
		\item the topological space $M$ is paracompact and Hausdorff; and
		\item for every $a\in M$, there is an open  neighborhood $U$ of $a$ in $M$ and $n,k\in \BN$ such that $(U, \CO|_U)$ is isomorphic to $(\R^n)^{(k)}$ as locally ringed spaces over $\mathrm{Spec}(\BC)$.
	\end{itemize}
\end{dfn}

By abuse of notation,  we will often not distinguish a formal manifold  $(M, \CO)$ with its underlying topological space $M$, and call $\CO$ the structure sheaf of it. 

Throughout this paper, let $(M,\CO)$ be a formal manifold. 
An element in $\CO(M)$ is called a formal function on $M$. By a chart of $M$, we mean a triple $(U,N^{(k)},\vartheta)$, where $U$ is an open subset of $M$, $N$ is an open submanifold  of $\BR^n$, $n,k\in \BN$, and 
\[\vartheta=(\overline{\vartheta},\vartheta^*):\  (U,\CO|_U)\rightarrow N^{(k)}
\]
is an isomorphism of locally ringed spaces. 
If there is no confusion, we will simply denote a chart $(U,N^{(k)},\vartheta)$ of $M$ by $U$. 
An open cover $\{U_\gamma\}_{\gamma\in \Gamma}$ of $M$ is called an atlas if for every $\gamma\in \Gamma$, $U_\gamma$ is a chart of $M$.

Let $a\in M$. The uniquely determined natural numbers $n$ and $k$ in Definition \ref{def:formalmanifold} are 
 called the dimension and the degree of $M$ at $a$, to be denoted by $\dim_a M$ and $\deg_a M$, respectively.
We let $\m_a$ denote the maximal ideal of the stalk $\CO_a$, and define the formal stalk  at $a$ to be  the local $\BC$-algebra
\be\label{eq:fs}
\widehat \CO_a:=\varprojlim_{k\in \BN} \CO_a/\m_a^k.
\ee
Similarly, let $\widehat{\m}_a$ denote the maximal ideal of $\widehat \CO_a$.
It is a fact  that 
\[
\widehat{\CO}_a\cong \C[[x_1,x_2,\dots,x_n,y_1,y_2,\dots,y_k]]\quad (n:=\dim_a M,\ k:=\deg_a M)
\]
as $\C$-algebras (see \cite[Proposition 2.6]{CSW1}). Then we obtain a formal manifold
\be \label{eq:M_a}M_a:=(\{a\},\widehat{\CO}_a).\ee
Here and below, we will often not distinguish a sheaf over a singleton with its space of global sections.

Let $\m_\CO$ denote the ideal of $\CO$ defined by
\be\label{eq:defmo}
\m_\CO(U):=\{f\in \CO(U)\mid f_a\in \m_a  \textrm{ for all $a\in U$}\},
\ee
where $U$ is an open subset of $M$, and $f_a$ is the germ of $f$ at $a$. 
Form the quotient sheaf  $\underline{\CO}:=\CO/{\m_{\CO}}$  over $M$. 
Then \be \label{eq:underM}\underline{M}:=(M,\underline{\CO})\ee is a smooth manifold,  called the reduction of $M$.  

In the rest part of the Introduction, let 
\[\varphi=({\overline\varphi}, \varphi^*):\ (M',\CO')\rightarrow (M,\CO)\] 
be a morphism of formal manifolds, and let $b\in M'$. 
By definition,  ${\overline\varphi}:  M'\rightarrow M$ is a continuous map and
\be\label{phin}
\varphi^*: \ {\overline\varphi}^{-1}\CO\rightarrow \CO'
\ee
is a $\BC$-algebra sheaf homomorphism that induces local homomorphisms on the stalks. 
For  open subsets $U$ of $M$ and $U'$ of $M'$ such that $\overline\varphi(U')\subset U$,
we write
\[\varphi^*_{U,U'}:\ \CO(U)\rightarrow \CO'(U')\] for the homomorphism of $\C$-algebras  induced by \eqref{phin}.
If there is no confusion, we will denote   $\varphi^*_{U,U'}$ by $\varphi^*_U$ or  $\varphi^*$ for simplicity.
We denote by 
\[
\varphi_b^*:\ \CO_{\overline{\varphi}(b)}\rightarrow \CO'_b\qquad \text{and}\qquad 
\varphi_b^*: \ \wh{\CO}_{\overline{\varphi}(b)}\rightarrow \wh{\CO'}_b \qquad (b\in M')
\]
the induced local homomorphisms on the stalks and formal stalks, respectively.

We have the obvious morphisms 
\be
\underline{M}\rightarrow M\qquad\text{and}\qquad M_a\rightarrow M\qquad (a\in M).
\ee
Furthermore, there are unique morphisms 
\be\label{eq:reductionvarphi}
\underline \varphi:\ \underline{M'}\rightarrow \underline{M}\qquad\text{and}\qquad \varphi_{b}:\ M'_{b}\rightarrow M_{\overline{\varphi}(b)}\ee   such that the diagrams
\be \label{eq:underlinecommdiag}
\begin{CD}
	M' @>  \varphi >> M\\
	@AAA          @AAA\\
	\underline{M'}@> \underline \varphi >>  \underline{M} \\
\end{CD} \qquad  \text{and}\qquad \begin{CD}
	M' @>  \varphi >> M\\
	@AAA          @AAA\\
	M'_{b}@> \varphi_{b} >> M_{\overline{\varphi}(b)} \\
\end{CD}
\ee
commute, respectively.

Let $n,k\in \BN$, and let $N$ be an open submanifold of $\BR^n$. Write 
 $x_1, x_2, \dots, x_n$ and $y_1,y_2,\dots,y_k$ for the standard coordinate functions and formal variables in $\CO_{N}^{(k)}(N)$, respectively. 
 We  call \[
 (x_1,x_2,\dots,x_n,y_1,y_2,\dots,y_k)\]
 the standard coordinate system of $N^{(k)}$. 
 When $M=N^{(k)}$, the morphism $\varphi: M'\rightarrow N^{(k)}$ is uniquely determined by the $(n+k)$-tuple (see Proposition \ref{prop:genebij})
 \[c_\varphi:=(\varphi^*(x_1),\varphi^*(x_2),\dots,\varphi^*(x_n),\varphi^*(y_1),\varphi^*(y_2),\dots,\varphi^*(y_k))\in (\CO'(M'))^{n+k}.\] 
In view of this, we  often formally say that $\varphi$ is given by 
 \[
 (x_1,x_2,\dots,x_n,y_1,y_2,\dots,y_k)\mapsto c_\varphi.
 \] 

\subsection{Local structure of  morphisms}
In Sections \ref{sec:inv} and \ref{sec:rankthm}, we mainly analyze the local structure of a morphism of formal manifolds using its differentials. As we will explain below,  
the situation is quite a bit more subtle than the analogous theory
of smooth manifolds.

Similar to the case of smooth manifolds,  a tangent vector at a point $a\in M$ is defined to be a derivation on the stalk $\CO_a$ (see Definition \ref{def:tangent}). 
 Write  $\mathrm{T}_a(M)$ for the space of tangent vectors of $M$ at $a$. 
 By taking the transpose of the  homomorphism $\varphi^*_b: \CO_{\overline{\varphi}(b)}\rightarrow \CO'_b$,  it induces  a linear map 
\[
d\varphi_b:\ \mathrm{T}_b(M')\rightarrow \mathrm{T}_{\overline{\varphi}(b)}(M),
\] which is called the differential of $\varphi$ at $b$. 


The most fundamental results on the local structure of smooth maps between smooth manifolds are the inverse function theorem (see \cite[Theorem 4.5]{L}) and its consequences, particularly the constant rank theorem (see \cite[Theorem 4.12]{L}). 
However, the following example shows that the inverse function theorem
cannot be directly generalized to the setting of formal manifolds.

  

\begin{exampled}\label{ex:intro1}
    Let $n,r\in \BN$ with $n\ge r$. We have a morphism 
   $  (\BR^{n-r})^{(r)}\rightarrow (\BR^{n})^{(0)} $
    such that the homomorphism 
    \[
\RC^\infty(\BR^{n})\rightarrow    \RC^\infty(\BR^{n-r})[[y_1, y_2, \dots, y_r]]
    \] of the $\BC$-algebras of formal functions 
    is given by the Taylor series expansion along $\BR^r$. 
For every $k\in \BN$, it induces an obvious morphism 
    \be \label{eq:introex1}
   (\BR^{n-r})^{(r+k)}= (\BR^{n-r})^{(r)}\times (\R^0)^{(k)}\rightarrow (\BR^{n})^{(k)}=(\BR^{n})^{(0)}\times (\R^0)^{(k)}.
    \ee 
  The differential of \eqref{eq:introex1} at every point of $\BR^{n-r}$ is a bijection, while the morphism \eqref{eq:introex1} is not a local isomorphism 
  unless $r=0$.
\end{exampled}

This new phenomenon leads to the following two natural questions. 

\begin{itemize}
\item[(Q1)] When is the morphism $\varphi$  an isomorphism near $b$?
\item[(Q2)] What is the local model
          of $\varphi$  when $d\varphi_b$ is bijective?
\end{itemize}

For the question (Q1), note that if $\varphi$ is an isomorphism near $b$ (and so is $\underline{\varphi}$), then both 
$d\varphi_b$ and $d\underline{\varphi}_b$ are bijective. Conversely, we have the following result.

\begin{thmd}\label{thm:invfun} Let 
$\varphi=({\overline\varphi}, \varphi^*): (M',\CO')\rightarrow (M,\CO)$  
be a morphism of formal manifolds, and let $b\in M'$. Assume that
\[d\varphi_b:\ \mathrm{T}_{b}(M')\rightarrow \mathrm{T}_{\overline\varphi(b)}(M)\quad \text{and}\quad
  d\underline{\varphi}_b:\ \mathrm{T}_{b}(\underline{M}')\rightarrow \mathrm{T}_{\overline\varphi(b)}(\underline{M}) \]  are  both bijective.
  Then there exists an open neighborhood  $U'$ of $b$ in $M'$ such that $U:=\overline\varphi(U')$ is an open subset of $M$, and the restriction$$\varphi|_{U'}:\ (U', \CO'|_{U'})\rightarrow (U, \CO|_{U}) $$ of $\varphi$ on $U'$ is an isomorphism of formal manifolds.
\end{thmd}	

When  $M$ and $M'$ are both smooth manifolds, Theorem \ref{thm:invfun} is just the usual inverse function theorem for smooth manifolds. When the underlying spaces of $M'$ and $M$ are both singletons, Theorem \ref{thm:invfun} is a reformulation of 
the formal inverse function theorem (see \cite[(A4.5)]{Ha} for example).   Theorem \ref{thm:invfun} is thus a common generalization of these two classical results.

The proof of Theorem \ref{thm:invfun} will be presented in \S\,\ref{subsec:invthm}. 
For the question (Q2),  
 we will show in Corollary \ref{cor:desbijdiff} that if $d\varphi_b$ is bijective, then the local expression 
 of $\varphi$  near $b$ is modeled by  the morphism \eqref{eq:introex1} for some $n,k,r\in \BN$. 

 Now we consider the case that the differential $d\varphi_b$ may or may not be bijective. 
As we have seen before, the local structure 
of $\varphi$ near $b$ depends not only on  $d\varphi_b$, but also on  $d\underline{\varphi}_b$. In fact, 
 $\mathrm{T}_b(\underline{M'})$ is naturally a subspace of $\mathrm{T}_b(M')$ and $d\underline{\varphi}_b$ coincides with the restriction of $d\varphi_b$ on $\mathrm{T}_b(\underline{M'})$ (see \S\,\ref{subsec:diff}).
 It is an elementary fact in linear algebra that such 
  a linear map $ d\varphi_b$ (namely, a linear map from $\mathrm{T}_{b}(M')$ to $  \mathrm{T}_{\overline\varphi(b)}(M)$ that maps $\mathrm{T}_b(\underline{M}')$ into $\mathrm{T}_{\overline\varphi(b)}(\underline{M}))$ is characterized by the triple (see Remark \ref{rem:char})
\be\label{eq:rankvarphib}(\mathrm{Rank}(d\varphi_b),\mathrm{Rank}(d\underline{\varphi}_b),\mathrm{Rank}(d\varphi_b^{\mathrm{f}}))\in \BN^3. \ee Here  
 $\mathrm{Rank}(\,\cdot\,)$ denotes the rank of a linear map between finite-dimensional vector spaces, and
\[
d\varphi_b^{\mathrm{f}}: \mathrm{T}_b(M')/\mathrm{T}_b(\underline{M'})\rightarrow
\mathrm{T}_{\overline{\varphi}(b)}(M)/\mathrm{T}_{\overline{\varphi}(b)}(\underline{M})
\]
denotes the quotient map induced by $d\varphi_b$.
 Taking this point of view, we introduce the following definition.
 
 \begin{dfn}\label{de:rank}
     The triple \eqref{eq:rankvarphib} is called 
     the rank of $\varphi$ at $b$, to be denoted by  $\mathrm{Rank}_b(\varphi)$.
 \end{dfn}

As in the smooth manifolds case, we say that $\varphi$ has constant rank near $b$ if there is an open neighborhood $V'$ of $b$ in $M'$ such that $\mathrm{Rank}_{b'}(\varphi)=\mathrm{Rank}_{b}(\varphi)$ for all $b'\in V'$. Motivated by the constant rank theorem for smooth manifolds,
we also make the following definition. 
 
\begin{dfn}\label{def:star} The morphism $\varphi$ is said to be standardizable near $b$ if there exists   a chart $(U',(N')^{(k')},\vartheta')$ of $M'$ containing $b$ and a chart $(U,N^{(k)},\vartheta)$ of $M$ containing $\overline{\varphi}(b)$ such that $\overline{\varphi}(U')\subset U$ and  that the morphism  
\be\label{eq:localrepnvarphi}
 \vartheta\circ   \varphi|_{U'}\circ (\vartheta')^{-1}:     \     (N')^{(k')}\rightarrow N^{(k)}
\ee
is given by 
\begin{equation}\begin{split}\label{eq:introstrandard}
&(x_1,\dots,x_{r_1},x_{r_1+1},\dots,x_{n-r_2},x_{n-r_2+1},\dots,x_n,y_1,\dots,y_{r_3},y_{r_3+1},\dots,y_k)\\
\mapsto\ & 
(u_1,\dots,u_{r_1},\underbrace{0,\dots,0}_{n-r_1-r_2},z_1,\dots,z_{r_2},z_{r_2+1},\dots,z_{r_2+r_3},\underbrace{0,\dots,0}_{k-r_3})
\end{split}\end{equation}
for some $r_1,r_2,r_3\in \BN$,
where $(x_1,x_2,\dots,x_n,y_1,y_2\dots,y_k)$ and $(u_1,u_2,\dots,u_{n'}, z_1,$ $z_2,\dots,z_{k'})$ are respectively the standard coordinate systems of $N^{(k)}$ and $(N')^{(k')}$. 
\end{dfn}

If $\varphi$ is standardizable near $b$ as in Definition \ref{def:star}, then $\varphi$ has constant rank  $(r_1+r_2+r_3,r_1,r_3)$ near $b$. However, if $\varphi$ has constant rank near $b$, it may not be standardizable near $b$, as shown in the following example.

\begin{exampled}
Suppose that  $M'=M=(\BR^0)^{(1)}$ (and then  $\CO(M)=\CO(M')=\BC[[y]]$), 
and $\varphi$ is determined by 
\[\varphi^*\left(\sum_{i\in \BN}c_i\,y^i\right)=\sum_{i\in \BN} c_i\,y^{2i}\quad (c_i\in \BC).\] It is obvious that $\varphi$ has constant rank $(0,0,0)$.  
If $\varphi$ is standardizable, then there exist  isomorphisms $\vartheta=(\overline{\vartheta},\vartheta^*):M\rightarrow M$ and $\vartheta'=(\overline{\vartheta'},\vartheta'^*):M'\rightarrow M'$ such that the kernel of the homomorphism \[((\vartheta')^{-1})^*\circ \varphi^*\circ (\vartheta)^*:\ \C[[y]]\rightarrow \C[[y]]\]   is   $y\BC[[y]]$, which contradicts to the fact that $\varphi^*$ is injective.
\end{exampled}

Naturally,  we have the following two questions.

\begin{itemize}
\item[(Q3)] What is the local model 
of $\varphi$ when it has constant rank near $b$?
\item[(Q4)] When  is the morphism $\varphi$ standardizable near $b$? 
\end{itemize}

By applying Theorem \ref{thm:invfun}, we will provide an answer of (Q3) in Theorem \ref{thm:rankthm}. 
For the question (Q4), let $b'\in M'$ and consider the homomorphism \be\label{eq:varphi_0}\varphi^*_{b',0}:\ \wh{\m}_{\overline{\varphi}(b')}\rightarrow \wh{\m'}_{b'}
\quad (\wh{\m'}_{b'}\ \text{is the maximal ideal of}\ \wh{\CO'}_{b'}) \ee
induced by $\varphi^*_{b'}:\wh{\CO}_{\overline{\varphi}(b')}\rightarrow \wh{\CO'}_{b'}$. 
  Then there are canonical homomorphisms  \be\label{eq:Introvarphi2} \varphi^*_{b',2}: \  \wh{\m}_{\overline{\varphi}(b')}/\wh{\m}_{\overline{\varphi}(b')}^2\rightarrow \wh{\m'}_{b'}/\wh{\m'}_{b'}^{2}\quad\text{and}\quad  \varphi_{b',\mathrm{k}}^*: \ker \varphi^*_{b',0}\rightarrow \ker \varphi^*_{b',2}\ee
         such that the diagram          
        \be\label{eq:constantrankthediagram} \begin{CD} \ker \varphi^*_{b',0}@>>>\wh{\m}_{\overline{\varphi}(b')} @>\varphi^*_{b',0}>>\wh{\m'}_{b'}\\@V \varphi_{b',\mathrm{k}}^* VV @VVV @VVV\\ \ker \varphi^*_{b',2}@>>>  \wh{\m}_{\overline{\varphi}(b')}/\wh{\m}_{\overline{\varphi}(b')}^2@>\varphi^*_{b',2}>> \wh{\m'}_{b'}/\wh{\m'}_{b'}^{2} \end{CD}\ee commutes. 
     We have the following result about (Q4), whose proof will be given in \S\,\ref{subsec:rankthm}.

    \begin{thmd}\label{thm:constantrankthmmain}Let $\varphi=(\overline\varphi, \varphi^*): (M',\CO')\rightarrow (M,\CO)$ be a morphism of formal manifolds, and let $b\in M'$. 
    Then the morphism $\varphi$ is standardizable near $b$ if and only if 
 there is an open neighborhood $V'$ of $b$ in $M'$ such that for every $b'\in V'$,  $\mathrm{Rank}_{b'}(\varphi)=\mathrm{Rank}_b(\varphi)$ and  
      $\varphi^*_{b',\mathrm{k}}$  is surjective.
\end{thmd}

In \S\,\ref{subsec:immvssub}, we will consider the special case that $\varphi$ is an immersion or a submersion at $b$, namely, $d\varphi_b$ is injective or surjective. 
In such cases, $\varphi$ may or may not have constant rank near $b$ (see Examples \ref{ex:immnotconstant} and \ref{ex:subnotconstant}). On the other hand, if  $\varphi$  is in addition assumed to have constant rank near $b$, then it is standardizable near $b$ (see Propositions  \ref{prop:localimmersion} and 
\ref{prop:dessubmer}). 
A sufficient condition for $\varphi$ to have  constant rank near $b$ is that it is regular at $b$ in the sense that
\[
\mathrm{Rank}(d\varphi_b)=\mathrm{Rank}(d\underline{\varphi}_b)+\mathrm{Rank}(d\varphi_b^{\mathrm{f}}).\] See  Lemma \ref{lem:regularmor} for details.

We remark that the local behavior of a regular morphism of formal manifolds is quite similar to that of a morphism of supermanifolds.
For example, there are  supergeometric versions of the inverse function,  immersion and submersion theorems for morphisms of supermanifolds (see \cite[Propositions 5.2.1, 5.2.4, and 5.2.5]{CCF}), and we 
have analog results for regular morphisms of formal manifolds (see Corollaries \ref{cor:desbijdiff}, \ref{cor:desregsumer}, and \ref{cor:desregimmer}). 
On the other hand, a constant rank regular morphism of formal manifolds may not be standardizable, and the same phenomenon appears in the supermanifolds case. 

\begin{remarkd} The obvious analog of Theorem \ref{thm:constantrankthmmain} for supermanifolds, which characterizes the standardizable morphisms between supermanifolds, also holds true. In \cite[\S\,5.2]{CCF}, a notion of constant rank morphisms between  supermanifolds is defined and the corresponding constant rank theorem (see \cite[Proposition 5.2.6]{CCF}) is proved, which gives another characterization of the standardizable morphisms between supermanifolds. We remark that, unlike to the smooth manifolds case,  constant rank morphisms in \cite{CCF} are not defined by  the ranks of the differential maps.  
 Additionally, the condition to character standardizable morphism in Theorem \ref{thm:constantrankthmmain} is given point-wisely, while that of \cite[Proposition 5.2.6]{CCF} is given locally but not point-wisely.
\end{remarkd}

\subsection{Formal submanifolds}
In Section \ref{sec:submanifold}, we aim to develop the theory of formal submanifolds. 
As usual, $\varphi$ is called an immersion if $d\varphi_{b'}$ is injective for every $b'\in M'$. Additionally,  $\varphi$ is said to be injective (resp.\,closed; resp.\,embedded)  if  $\overline\varphi$ is injective (resp.\,closed; resp.\,a topological embedding).
 For examples,  the morphisms  $\underline{M}\rightarrow M$ and $M_a\rightarrow M$ ($a\in M$) are both closed embedded immersions (see Lemma \ref{lem:taninj}). We have the following characterizations of immersions and closed embedded immersions of formal manifolds, whose proof will be given in \S\,\ref{subsec:charcloembd}.

\begin{thmd}\label{thm:charclosedsub}
 Let
  $\varphi=(\overline\varphi,\varphi^*): (M',\CO')\rightarrow (M,\CO)$ be a morphism of formal manifolds.

\noindent(a) The morphism $\varphi$ is an  immersion if and only if the $\C$-algebra sheaf homomorphism 
\[\varphi^*:\ \overline{\varphi}^{-1}\CO\rightarrow \CO'\] 
is surjective.

\noindent(b) The morphism $\varphi$ is a closed embedded immersion if and only if the $\C$-algebra homomorphism 
\[\varphi^*:\ \CO(M)\rightarrow \CO'(M')\]
is surjective. 
\end{thmd}



Let $\mathrm{Imm}(M)$ denote the family of all pairs $(M_0,\varphi_0)$, where $M_0$ is a formal manifold and $\varphi_0$ is an injective immersion from $M_0$ to $M$. Two pairs  $(M_1,\varphi_1)$ and $(M_2,\varphi_2)$ in $\mathrm{Imm}(M)$ are said to be equivalent if there exists an isomorphism 
 $\psi:M_1\rightarrow M_2$ of formal manifolds such that $\varphi_2\circ \psi=\varphi_1$. As in the case of smooth manifolds, an immersed formal submanifold of $M$ defined below can be understood as an equivalence class in $\mathrm{Imm}(M)$.

\begin{dfn} An immersed formal submanifold of $(M,\CO)$ is a pair  $(S,(i^{-1}\CO)/\CI)$ consisting of
\begin{itemize} 
\item a topological space $S$   such that $S$ is a subset of $M$ and the inclusion $i:S\rightarrow M$ is continuous; and \item a quotient sheaf $(i^{-1}\CO)/\CI$ over  $S$, where $\CI$ is an ideal  of the sheaf $i^{-1}\CO$ of $\BC$-algebras over $S$,  
\end{itemize} such that, as a locally ringed space, $(S,(i^{-1}\CO)/\CI)$ is a formal manifold. 
\end{dfn}

For each immersed formal submanifold $(S,(i^{-1}\CO)/\CI)$ of $(M,\CO)$, we have  a canonical morphism \be\label{eq:mapiota} \iota=(\overline\iota,\iota^*):\ (S,(i^{-1}\CO)/\CI)\rightarrow (M,\CO)\ee of formal manifolds, where $\overline{\iota}=i$ is the inclusion, and for every open subset $U$ of $M$, $\iota^*_{U}$ is the composition map \be \label{eq:iotaU}\iota^*_{U}:\ \CO(U)\rightarrow (i^{-1}\CO)(i^{-1}(U))\rightarrow ((i^{-1}\CO)/\CI)(i^{-1}(U)).\ee 
Then $(S,\iota)$ is an element in $\mathrm{Imm}(M)$.
Conversely, each equivalence class in $\mathrm{Imm}(M)$  has a unique representative of the form $(S,\iota)$ (see Proposition \ref{prop:dessub}).

By abuse of notation, we will often not distinguish an immersed formal submanifold $(S,(i^{-1}\CO)/\CI)$  of $(M,\CO)$  
with its underlying topological space $S$, and call $\CI$ the defining ideal of it.

 An immersed formal submanifold $S$ of $M$ is called an embedded formal submanifold if $S$ is equipped with the subspace topology of $M$. 
In the category of smooth manifolds, we also define immersed smooth submanifolds and embedded smooth submanifolds similarly. 

Given an immersed formal submanifold $S$ of $M$, a fundamental question is in which case it has the universal property that for every morphism $\psi=(\overline{\psi},\psi^*): M'\rightarrow M$  with $\overline{\psi}(M')\subset S$,  there is a unique morphism $\phi: M'\rightarrow S$ such that $\psi=\iota\circ \phi$?
In the category of smooth manifolds, an immersed smooth submanifold that admits such a universal property is usually called initial (see \cite[Page 114]{L}). It is well-known that each embedded smooth submanifold is initial (see \cite[Corollary 5.30]{L}).

If $S$ is an initial formal submanifold of $M$ in the sense that it satisfies the universal property mentioned above, then $\underline{S}$ is an initial smooth submanifold of $\underline{M}$. On the other hand, the following theorem shows that for each initial smooth submanifold $P$ of $\underline{M}$, there is a unique initial formal submanifold of $M$ whose reduction is $P$.

\begin{thmd} \label{thm:universalformalsub}
Let $P$ be an immersed smooth submanifold of $\underline{M}$. Then there is a unique immersed formal submanifold $\wt P$ of $M$ such that  $\wt P=P$ as topological spaces and satisfies the following universal property:
 for every morphism $\psi=(\overline{\psi},\psi^*): M' \rightarrow M$ of formal manifolds, if  
$\overline{\psi}(M')\subset P$ and the induced map $\overline{\psi}: M'\rightarrow P$  is continuous,  
then there exists a unique  morphism $\phi=(\overline\phi,\phi^*): M' \rightarrow \wt P$  such that  the diagram 
\[\xymatrix{ & \wt P\ar[d]^{\iota}\\
  M'\ar[r]^{\psi}\ar[ur]^{\phi}& M}
  \]
  commutes.
\end{thmd}

The proof of Theorem \ref{thm:universalformalsub} is presented in \S\,\ref{subsec:proofthmuniformalsub}, and we will also give some consequences of it in \S\,\ref{subsection:app}. 

For every $a\in M$, denote by 
\be \label{eq:varsigma}\varsigma_a=(\overline{\varsigma_a},\varsigma_a^*):\ \Spec(\BC)\rightarrow M \ee the unique morphism  that $\overline{\varsigma_a}$ maps the singleton in $\Spec(\BC)$ to $a$.
In the case of smooth manifolds,  many smooth submanifolds are presented as level sets of constant-rank smooth maps (see \cite[Theorem 5.12]{L}). 
We have the following generalization whose proof will be given in \S\,\ref{subsec:levelsets}. 

\begin{thmd}\label{thm:levelsets} Let $\varphi=(\overline\varphi, \varphi^*): (M',\CO')\rightarrow (M,\CO)$ be a morphism of formal manifolds, and let $a\in M$. Assume that the morphism $\varphi$ is standardizable near $b$ for every $b\in \overline\varphi^{-1}(a)$. Then there exists a Cartesian diagram 
\[\begin{CD}\varphi^{-1}(a) @>>>\Spec (\BC)\\ @V\iota VV @VV\varsigma_a V \\ M'@>\varphi>> M \end{CD}\] in the category of formal manifolds.
Moreover, $\iota$ is a closed embedded immersion. 
\end{thmd}


\section{Tangent spaces, differentials, and inverse function theorem}\label{sec:inv}
In this section, we introduce and study the notion of tangent spaces of a formal manifold, as well as the notion of differentials of a morphism. 
We also present a proof of Theorem \ref{thm:invfun}.

 \subsection{Tangent spaces}\label{subsec:tangentspace}
Let $a\in M$. Form the quotient maps
\begin{eqnarray*}
\label{eq:defdeltax}
\delta_a:\ \mathcal{O}_a\longrightarrow \mathcal{O}_a/\m_a=\C 
\quad\text{and}\quad \delta_a:\  \widehat{\mathcal{O}}_a\longrightarrow \widehat{\mathcal{O}}_a/\widehat{\m}_a=\C.\end{eqnarray*}
For every $f$ in $\CO_a$ or $\wh{\CO}_a$, write \[f(a):=\delta_a(f)\in \BC.\]
Additionally, for every formal function $g$ on $M$,
 we define the value  \[g(a)\in \BC\] of $g$ at $a$ to be
the image of $g$ under the following composition of natural maps:
\be\label{eq:Ev_a2} \mathrm{Ev}_a: \ \mathcal{O}(M)\longrightarrow \mathcal{O}_a\stackrel{\delta_a}{\longrightarrow} \C.\ee
 
As mentioned in the Introduction, we make the following definition. 

\begin{dfn}\label{def:tangent}
Let $a\in M$.
A tangent vector of $M$ at $a$
is a linear functional $\eta: \CO_a\rightarrow \BC$ such that
\[
\la \eta, fg\ra=f(a)\cdot \la \eta, g\ra+ \la \eta, f\ra \cdot g(a)\quad \textrm{for all }f,g\in \CO_a.
\]   
\end{dfn}

Recall that  $\mathrm{T}_a(M)$  denotes the tangent space of $M$ at $a$, which consists of all tangent vectors of $M$ at $a$. When $M$ is a smooth manifold,  $\mathrm T_a(M)$ is the complexification of the usual tangent space.

\begin{lemd}\label{lem:dimtan} The tangent space  $\mathrm{T}_a(M)$ is naturally  isomorphic to
the dual space $(\m_a/\m_a^2)^*$,
and 
\be \label{eq:dimoftangentspace}\dim \mathrm T_a (M)=\dim_a M+\deg_a M.\ee 
\end{lemd}	

\begin{proof} For every $\eta\in \mathrm{T}_a (M)$, the restriction of $\eta$ on $\m_a$ induces a linear functional  $\overline{\eta}$ on $\m_a/\m_a^2$.
One easily checks that the assignment 
$\eta\mapsto \overline{\eta}$ defines a linear isomorphism from  $\mathrm{T}_a(M)$  to $(\m_a/\m_a^2)^*$.
On the other hand, by 
\cite[Lemma 2.5]{CSW1}, we have that
\[\dim \m_a/\m_a^2=\dim_a M+\deg_a M.\] Hence $\dim \mathrm T_a (M) =\dim_a M+\deg_a M$.
\end{proof}

\begin{remarkd}\label{rem:hatTaM}
  It follows from the proof of \cite[Proposition 2.6]{CSW1} that the natural map $\m_a/(\m_a)^2\rightarrow \wh{\m}_a/(\wh{\m}_a)^2$  is an isomorphism. Thus we also have the   identification
  \[
  \mathrm{T}_a(M)=(\wh{\m}_a/(\wh{\m}_a)^2)^*.
  \]
\end{remarkd}

Let $n,k\in \BN$, and let $N$ be an open submanifold of $\BR^n$. Write 
\[(x_1,x_2,\dots,x_n,y_1,y_2,\dots,y_k)\] for the 
 the standard coordinate system of $N^{(k)}$.
  We denote the first-order partial derivatives with respect to the variables $x_1, x_2, \dots, x_n$ as
\[\partial_{x_1},\partial_{x_2},\dots,\partial_{x_n}.\]
 For each $n$-tuple $I=(i_1,i_2,\dots,i_n)\in \BN^n$, we use  the following usual multi-index notations:
\[  |I|=i_1+i_2+\cdots+i_n,\quad 
  x^I=x_1^{i_1}x_2^{i_2}\cdots x_n^{i_n},  \qaq \partial_x^I=\partial_{x_1}^{i_1}\partial_{x_2}^{i_2}\cdots \partial_{x_n}^{i_n}.
\]
We also view the $|I|$-th order partial derivative $\partial_x^I$ as an operator on the formal power series algebra $\RC^\infty(N)[[y_1,y_2,\dots,y_k]]$ in
an obvious way.
Let
\[\partial_{y_1},\partial_{y_2},\dots,\partial_{y_k}\] denote the first-order  partial  derivatives on  $\RC^\infty(N)[[y_1,y_2,\dots,y_k]]$
with respect to the formal variables $y_1,y_2,\dots,y_k$.
For every $k$-tuple $J\in \BN^k$, 
the length $|J|$, the monomial $y^J$, and the $|J|$-th order partial derivative $\partial_y^J$, are defined in a similar way.
Note that derivatives $\partial^I_x$ and $\partial^J_y$  act naturally on the stalks and formal stalks of $N^{(k)}$.

\begin{exampled}\label{ex:tanvec}
With the above notations,  the set
\[\{\delta_a\circ \partial_{x_i},\, \delta_a\circ\partial_{y_j}\mid i=1,2,\dots,n,\,j=1,2,\dots,k\}\]
forms  a basis of $\mathrm{T}_a (N^{(k)})$.
\end{exampled}

 By an LCS, we mean a locally convex topological vector
space over $\BC$ which may or may not be Hausdorff. 
For an LCS $E$, let $E'$ denote the space of all continuous linear functionals on $E$, equipped with the strong topology.

 We equip the algebra  $\CO(M)$  with the smooth topology (see \cite[Definition 4.1]{CSW1}),
 equip the stalk $\CO_a$ with the canonical inductive limit topology (see \cite[(5.12)]{CSW2}), and equip the formal stalk
$\widehat\CO_a$ with the canonical projective limit topology (see \cite[\S\,2.3]{CSW1}). 
Then all of them are LCS, and $\widehat{\CO}_a$ is in fact the Hausdorff LCS associated to $\CO_a$, namely, $\CO_a/\overline{\{0\}}=\wh\CO_a$ as LCS (see \cite[Lemma 5.25]{CSW2}).

Write $\mathrm{Dist}_a(M)\subset (\CO(M))'$ for the LCS of all formal distributions on $M$ supported in $\{a\}$. 
By taking the transpose, the canonical continuous linear maps 
\[
\CO(M)\rightarrow \CO_a\quad\text{and}\quad \CO_a\rightarrow \wh\CO_a
\]
induce the following LCS identifications:
\begin{eqnarray}\label{eq:desDist}
  (\wh\CO_a)'=(\CO_a)'=\mathrm{Dist}_a(M).  
\end{eqnarray}
See  \cite[\S\,5.3]{CSW2} for details. 

Example \ref{ex:tanvec} and \cite[Lemma 5.23]{CSW2} imply that every tangent vector of $M$ at $a$ is continuous. Thus we can identify $\mathrm{T}_a(M)$ as a subspace of $(\widehat{\CO}_a)'$ or 
$\mathrm{Dist}_a(M)$. In fact, write $\mathrm{Der}_a(M)$ (resp.\,$\widehat{\mathrm{T}}_a(M)$) for the space of those $\BC$-valued linear functionals $\eta$ on $\CO(M)$ (resp.\,$\widehat{\CO}_a$) such that 
\begin{eqnarray*}
\la\eta, fg\ra=f(a)\cdot\la \eta, g\ra+ \la \eta, f\ra \cdot g(a)\quad \text{for all}\ f,g\in \CO(M)\ \text{(resp.\,$\widehat{\CO}_a$)}.
\end{eqnarray*}

The following result is straightforward. 

\begin{lemd}\label{lem:idenTaM} Under the identification   \eqref{eq:desDist},  we have that 
\[
\widehat{\mathrm{T}}_a(M)=\mathrm{T}_a(M)=\mathrm{Der}_a(M).
\]
\end{lemd}

\subsection{Differentials} \label{subsec:diff}

Let
\be\label{eq:morvarex}
\varphi=(\overline\varphi,\varphi^*):\ (M',\CO')\rightarrow (M,\CO)
\ee be a morphism of formal manifolds, and let $b\in M'$.
For every $\eta\in \mathrm{T}_b (M')$, the linear functional
 \be\label{eq:dvarphieta} d\varphi_b(\eta):\ \CO_{\overline\varphi(b)}\rightarrow \C,\qquad f\mapsto  \la\eta, \varphi_b^*(f)\ra\ee 
 is  a tangent vector of $M$ at $\overline\varphi(b)$.
As in the Introduction, the resulting linear map
\be
\label{eq:diffmapata}d\varphi_b:\ \mathrm{T}_b (M')\rightarrow \mathrm{T}_{\overline\varphi(b)}(M),\qquad
\eta\mapsto d\varphi_b(\eta)\ee
 is called the differential of $\varphi$ at $b$. 
It is obvious that
\be\label{eq:chainrule} d(\varphi'\circ \varphi)_b=d\varphi'_{\overline{\varphi}(b)}\circ d\varphi_b,\ee 
where $\varphi': M\rightarrow M_1$ is another morphism of formal manifolds.
 When both $M'$ and $M$ are smooth manifolds, $d\varphi_b$ is the usual differential map.

In view of \cite[Theorem 4.15 and Lemma 2.8]{CSW1}, the $\C$-algebra homomorphisms
\[
\varphi^*:\ \CO(M)\rightarrow \CO'(M'),\quad  \varphi^*_b:\ \CO_b\rightarrow \CO'_b,\quad\text{and}\quad 
\varphi^*_b:\ \widehat{\CO}_b\rightarrow \widehat{\CO'}_b
\]
are all continuous. 
By taking the  transpose of the continuous linear map 
\be \label{eq:convarphia}
\varphi_b^*:\ \CO'_{\overline\varphi(b)}\rightarrow \CO_b \quad\text{or}\quad \varphi^*_b:\ \widehat {\CO'}_{\overline{\varphi}(b)}\rightarrow \widehat \CO_b,
\ee we obtain  a linear map (see \eqref{eq:desDist})
 \be \label{eq:tvarphia}{}^t\varphi_b^*:\ \mathrm{Dist}_{b}(M)\rightarrow \mathrm{Dist}_{\overline\varphi(b)}(M').\ee
It is clear that \be \label{eq:diffmapasres}
{}^t\varphi_b^*|_{\mathrm{T}_b(M)}=d\varphi_b:\ 
\mathrm{T}_b(M)\rightarrow \mathrm{T}_{\overline\varphi(b)}(M').\ee

\begin{remarkd}\label{rem:dvarphivsvarphi2}
   Using the identification given in Remark \ref{rem:hatTaM}, $d\varphi_b$ coincides with the transpose of the linear map (see \eqref{eq:Introvarphi2})
    \[
    \varphi^*_{b,2}: \  \wh{\m}_{\overline{\varphi}(b)}/\wh{\m}_{\overline{\varphi}(b)}^2\rightarrow \wh{\m'}_{b}/\wh{\m'}_{b}^{2}.
    \]
\end{remarkd}

Recall the formal manifolds $M_a$ and $\underline{M}$ from  \eqref{eq:M_a} and \eqref{eq:underM}, respectively.
\begin{lemd}\label{lem:taninj}
For every $a\in M$, the differential 
\be\label{eq:taninj}
\mathrm{T}_{a} (\underline M)\rightarrow \mathrm{T}_{a} (M) \ee
of the  morphism $\underline{M}\rightarrow M$ at $a$ is an injection, and the differential
\[\mathrm{T}_{a} (M_a)\rightarrow \mathrm{T}_{a} (M)
\]
of the morphism $M_a\rightarrow M$ at $a$ is a bijection. 
\end{lemd} 
 \begin{proof}
 Since the homomorphism 
\[\CO(M)\rightarrow \underline{\CO}(M)
\] 
 is surjective (see \cite[Proposition 2.14]{CSW1}), 
the linear map $\mathrm{Dist}_{a}(\underline M)\rightarrow \mathrm{Dist}_{a}(M)$ is injective and so is the differential map \eqref{eq:taninj}. The second assertion follows easily from the identification that $\mathrm{T}_a(M)=\widehat{\mathrm{T}}_a(M)$ (see Lemma \ref{lem:idenTaM}).
 
 \end{proof}



Identify  $\mathrm{T}_{a}(\underline M)$ as a subspace of $\mathrm{T}_{a}( M) $ via the injection \eqref{eq:taninj}, and set 
\[
\mathrm{T}_a^\mathrm{f}(M):=\mathrm{T}_a(M)/\mathrm{T}_a(\underline{M}).
\]
Let $\varphi$ be a morphism as in \eqref{eq:morvarex}.
By \eqref{eq:underlinecommdiag} and 
\eqref{eq:chainrule}, we have the commutative diagram
\be\begin{CD}
\mathrm{T}_b(\underline{M'})@>d\underline{\varphi}_b>> \mathrm{T}_{\overline{\varphi}(b)}(\underline{M})\\ @VVV @VVV \\\mathrm{T}_b(M')@>d\varphi_b>> \mathrm{T}_{\overline{\varphi}(b)}(M).
\end{CD}\ee
Then we have a linear map 
\begin{eqnarray} \label{eq:dulphi}
d\varphi_b|_{\mathrm{T}_{b} (\underline M')}=d\underline{\varphi}_b:\ \mathrm{T}_{b}(\underline M')\rightarrow  \mathrm{T}_{\overline\varphi(b)}(\underline M).
\end{eqnarray}
This also induces a natural linear map 
\[
d\varphi_b^\mathrm{f}:\ \mathrm{T}_b^\mathrm{f}(M')\rightarrow \mathrm{T}_{\overline{\varphi}(b)}^\mathrm{f}(M).
\]

\begin{remarkd}\label{rem:char} Let $E$ and $E'$ be two finite-dimensional complex vector spaces. Up to equivalence, linear maps from $E$ to $E'$ are characterized by their ranks. 
Let $E_1$ and $E_1'$ be subspaces of $E$ and $E'$,  respectively. Here we give an analog characterization of the set
\be \label{eq:linearmapsh}
\{h:E'\rightarrow E \mid \text{$h$ is a linear map satisfying that $h(E_1)\subset E_1'$}\}.
\ee
For every linear map $h$ as in \eqref{eq:linearmapsh}, we have two canonical linear maps
\[
\underline{h}:\ E'_1\rightarrow E_1\quad\text{and}\quad 
h^{\mathrm{f}}:\ E'/E'_1\rightarrow E/E_1
\]
induced by $h$. 
On the other hand, let $h_1$ and $h_2$ be two linear maps in \eqref{eq:linearmapsh}.
  We say that  $h_1$ is equivalent to $h_2$  if 
  there is a linear automorphism  $\vartheta$ on $E$ and a linear automorphism $\vartheta'$ on $E'$  such that \[\vartheta(E_1)=E_1,\quad \vartheta'(E_1')=E_1,\quad\text{and}\quad \vartheta\circ h_1\circ \vartheta'=h_2.\]
It is straightforward to verify that the linear maps $h_1$ and $h_2$ are equivalent if and only if 
\[(\mathrm{Rank}(h_1),\mathrm{Rank}(\underline {h_1}),\mathrm{Rank}(h_1^{\mathrm{f}}))=(\mathrm{Rank}(h_2),\mathrm{Rank}(\underline {h_2}),\mathrm{Rank}(h_2^{\mathrm{f}})).\]
In other words, as mentioned in the Introduction, a linear map $h$ in \eqref{eq:linearmapsh} is characterized by the triple
$(\mathrm{Rank}(h),\mathrm{Rank}(\underline {h}),\mathrm{Rank}(h^{\mathrm{f}}))$.
\end{remarkd}

Motivated by the above remark, 
we make the following definition (see also Definition \ref{de:rank}).

\begin{dfn}\label{de:rankvarphia}
We call   the triple 
    \[
    (\mathrm{Rank}(d\varphi_b),\mathrm{Rank}(d\underline{\varphi}_b), \mathrm{Rank}(d\varphi_b^\mathrm{f}))
    \]
  the rank of $\varphi$ at $b$, to be denoted by  $\mathrm{Rank}_b(\varphi)$.
\end{dfn}

Note that we have the following obvious inequality: 
 \[
\mathrm{Rank}(d\varphi_b)\ge \mathrm{Rank}(d\underline{\varphi}_b)+ \mathrm{Rank}(d\varphi_b^\mathrm{f}).
 \]

\begin{exampled}\label{ex:Jacobian} Suppose that $N'$ is an open submanifold of $\R^{n'}$, $N$ is a nonempty open submanifold of $\R^n$, and
\[\varphi=(\overline\varphi,\varphi^*):{(N')}^{(k')}\rightarrow {N}^{(k)}\] is a morphism of formal manifolds, where $n,n',k,k'\in \BN$.  Let 
\[(u_1, u_2, \dots , u_{n'},z_1,z_2,\dots,z_{k'})\quad\text{and}\quad (x_1, x_2, \dots , x_n,y_1,y_2,\dots,y_k)\]
denote the standard coordinate systems of $(N')^{(k')}$ and $N^{(k)}$, respectively. 
For $i=1,2,\dots,n$ and $j=1,2,\dots,k$, write
\begin{eqnarray}
\varphi^*(x_i)&=&f_i+\sum_{j'=1}^{k'} g_{ij'} z_{j'}+\sum_{J\in \BN^{k'}, |J|\ge 2} g_{iJ} z^J,\\
\varphi^*(y_j)&=&\sum_{j'=1}^{k'} h_{jj'} z_{j'}+\sum_{J\in \BN^{k'}, |J|\ge 2} h_{jJ} z^J,
              \end{eqnarray}
              where $f_i, g_{ij'}, g_{iJ}, h_{jj'}, h_{jJ}\in \RC^\infty(N')$. One easily checks that for every $b\in N'$,
\begin{eqnarray*}
d\varphi_b(\delta_b\circ \partial_{u_{i'}})&=&\sum_{i=1}^n  (\partial_{u_{i'}}f_i)(b)(\delta_{\overline\varphi(b)}\circ \partial_{x_i}),\\
d\varphi_b(\delta_b\circ \partial_{z_{j'}})&=&\sum_{i=1}^n  g_{ij'}(b)(\delta_{\overline\varphi(b)}\circ \partial_{x_i})
+\sum_{j=1}^k h_{jj'}(b)(\delta_{\overline\varphi(b)}\circ \partial_{y_j}),\end{eqnarray*}
where $i'=1,2,\dots,n'$ and $j'=1,2,\dots,k'$.
              We call the $(n+k)\times (n'+k')$-matrix
\begin{equation}\label{eq:Jacobian}
J=\left(
  \begin{array}{cc}
    F & G \\
    0 & H \\
  \end{array}
\right) \end{equation}
 the Jacobian  of $\varphi$, where
\[ F=\left(\partial_{u_{i'}}(f_i)\right)_{\substack{1\le i\le n,\\1\le i'\le n'}},\
G=\left(g_{ij'}\right)_{\substack{1\le i\le n,\\1\le j'\le k'}},\quad \text{and}\quad  H=\left(h_{jj'}\right)_{\substack{1\le j\le k,\\1\le j'\le k'}}.\]
Then $F$ is the usual Jacobian of the smooth map $\underline{\varphi}: N'\rightarrow N$. 
For every $b\in N'$, the matrix $J_b$ is called the Jacobian of $\varphi$ at $b$, and we have that 
\be \label{eq:Rk=JFH}
\mathrm{Rank}_b(\varphi)=(\mathrm{Rank}(J_b),\mathrm{Rank}(F_b),\mathrm{Rank}(H_b)).
\ee
Furthermore, the following statements are equivalent to each other:
\begin{itemize}
\item $d\varphi_b$ is a bijection;
\item  there is an open neighborhood $U'$ of $b$ such that  $d\varphi_{b'}$ is a bijection for all $b'\in U'$;
\item $J_b$ is invertible.
\end{itemize}
Here and henceforth,  for a matrix $C=(c_{ij})$ of smooth functions on $N'$, $C_b$ stands for the complex matrix $(c_{ij}(b))$ obtained
by taking the values at $b$, and $\mathrm{Rank}(C_b)$ stands for the rank of $C_b$.

\end{exampled}

\subsection{Proof of Theorem \ref{thm:invfun}}\label{subsec:invthm}




This subsection is devoted to a proof of Theorem \ref{thm:invfun}. 
For every $r\in \BN$,
write $\m^r_{\CO}$ for the ideal of $\CO$  such that
\[\m_{\CO}^r(U)= \{f\in \CO(U): f_a\in \m^r_a \textrm { for all   } a\in U\} \]
for all  open subsets $U$ of $M$, where  $f_a$
is the germ of $f$ in $\CO_a$.

\begin{exampled}\label{ex:mor}
If $M=N^{(k)}$ for some smooth manifold $N$ and $k\in \BN$, then
\be
\label{eq:mor}\m_{\CO}^r(U)=\left\{\sum_{J\in \BN^k,|J|\ge r}f_J y^J\mid f_J\in \RC^\infty(U)\right\}.\ee
\end{exampled}


\begin{prpd}\label{prop:moge} We have an identification 
\be \label{eq:o=invlim}\CO=\varprojlim_{r\in \BN}\CO/\m_{\CO}^r\ee
of sheaves of $\mathbb C$-algebras. (More precisely, the canonical morphism $\CO\rightarrow \varprojlim\limits_{r\in \BN}\CO/\m_{\CO}^r$ is an isomorphism.) 
\end{prpd}
\begin{proof}
Take an atlas $\{U_\gamma\}_{\gamma\in \Gamma}$ of $M$.
For every $\gamma\in \Gamma$,
it follows from the softness of $\m^r_{\CO}$ (see  \cite[Corollary 2.13]{CSW1}), \cite[Theorem 9.9]{Br2}, and
\eqref{eq:mor} that
\begin{eqnarray*}
   && (\varprojlim_{r\in \BN}(\CO/\m_{\CO}^r))(W_\gamma)=\varprojlim_{r\in \BN}((\CO/\m_{\CO}^r)(W_\gamma))\\&=&\varprojlim_{r\in \BN}(\CO(W_\gamma)/\m_{\CO}^r(W_\gamma))=\CO(W_\gamma),
\end{eqnarray*}
where $W_\gamma$ is an open subset of $U_\gamma$. Thus we have that
\[\CO|_{U_\gamma}=(\varprojlim_{r\in \BN}(\CO/\m_{\CO}^r))|_{U_\gamma},\]
which implies the proposition.
\end{proof}

For every $r\in \BN$, form the quotient sheaf $\m_\CO^r/\m_{\CO}^{r+1}$ over $M$.
The softness of $\m_\CO^{r+1}$ and \cite[Theorem 9.9]{Br2} implies that
\[(\m_\CO^r/\m_{\CO}^{r+1})(U)=\m_\CO^r(U)/\m_{\CO}^{r+1}(U)\quad (\text{$U$ is an open subset of $M$}).\]
Then $\m_\CO^r/\m_{\CO}^{r+1}$ is naturally   a sheaf of  $\underline{\CO}$-modules with the action given by
\[(f+\m_{\CO}(U))\cdot  (g+\m_{\CO}^{r+1}(U))=fg+\m_{\CO}^{r+1}(U), \]
where  $f\in \CO(U)$ and $g\in \m_{\CO}^r(U)$.

\begin{exampled}  If $M=N^{(k)}$ for some smooth manifold $N$ and $k\in \BN$, then
the $\underline{\CO}(U)$-module $(\m^r_{\CO}/\m^{r+1}_{\CO})(U)$ is free with a basis
\be\label{eq:basismrr+1}\{y^J+\m^{r+1}_{\CO}(U): J\in \BN^k\ \text{with}\ |J|=r\}\ee  given by (see \eqref{eq:mor})
\end{exampled}

Let $\varphi=(\overline\varphi,\varphi^*): (M',\CO')\rightarrow (M,\CO)$ be a morphism of formal manifolds.
For every $r\in \BN$, there is a unique morphism
\[
\varphi_r=(\overline{\varphi},\varphi^*_r):\ (M',{\CO'}/\m_{\CO'}^{r}) \rightarrow (M,\CO/\m_{\CO}^{r})\]
of locally ringed spaces such that the diagram
\begin{equation*} \xymatrix{
 (M',\CO')
   \ar[rr]^{\varphi}
                &&(M,\CO) \\
 (M',\CO'/\m_{\CO'}^{r})\ar[u]  \ar[rr]^{\varphi_r}
                && (M,\CO/\m_{\CO}^{r}) \ar[u]           }
\end{equation*}
commutes. 
Moreover, $\varphi_{r+1}$ induces a morphism
\[\varphi^*_{r,r+1}:\ \overline\varphi^{-1}(\m_{\CO'}^r/\m_{\CO'}^{r+1})\rightarrow \m_\CO^r/\m_{\CO}^{r+1}\]
 of $ \overline{\varphi}^{-1}\underline{\CO}$-modules.

\begin{exampled}\label{ex:varphir} In the setting of  Example \ref{ex:Jacobian},
take open subsets $W\subset N$ and $W'\subset N'$  such that  $\overline\varphi(W')\subset W$.
Let $r\in \BN$. Every element in $(\m_{\CO_N^{(k)}}^{r}/\m_{\CO_N^{(k)}}^{r+1})(W)$
is uniquely  of the form (see \eqref{eq:basismrr+1})
\[\mathbf{f}:=\sum_{J\in \BN^k, |J|=r} f_J y^J+\m_{\CO_N^{(k)}}^{r+1}(W), \]
where  $f_J\in \RC^\infty(W)$.
It is routine to check that
 \begin{eqnarray*}
(\varphi^*_{r,r+1})_{W,W'}(\mathbf{f})
=\sum_{J\in \BN^k, |J|=r}\underline{\varphi}^*_{W,W'}(f_J)\cdot H(z^J)|_{W'}+\m_{\CO_{N'}^{(k')}}^{r+1}(W'),
\end{eqnarray*}
where for $J=(j_1,j_2,\dots,j_k)\in \BN^k$,
\[H(z^J):=H(z_1)^{j_1}H(z_2)^{j_2}\cdots H(z_k)^{j_k}\quad \text{with}\quad H(z_j):=\sum_{j'=1}^{k'} h_{jj'}\cdot z_{j'}.\]
This particularly implies that if the reduction $\underline{\varphi}$ is a diffeomorphism and the matrix $H$ is invertible, then
$\varphi_{r,r+1}^*$ is an isomorphism.
\end{exampled}

Now we begin to prove Theorem \ref{thm:invfun} in the following special case.
\begin{lemd}\label{lem:invfun} Let $N',N$ be two open subsets of $\R^n$, and let \[\varphi=(\overline\varphi,\varphi^*):\ (N')^{(k)}\rightarrow
N^{(k)}\] be a morphism, where $n,k\in \BN$.
Assume that $\varphi$ satisfies the following conditions:
\begin{itemize}
\item the continuous  map $\overline{\varphi}: N'\rightarrow N$ is bijective; and
\item the differential $d\varphi_b:\mathrm{T}_b ((N')^{(k)})\rightarrow \mathrm{T}_{\overline\varphi(b)}(N^{(k)})$ is bijective for all $b\in N'$.
\end{itemize}
Then $\varphi$ is an isomorphism of formal manifolds.
\end{lemd}
\begin{proof}  With the notations used in Example \ref{ex:Jacobian}, the second condition in the lemma implies that for every $b\in N'$,
the Jacobian $J_b$ of $\varphi$ at $b$ is invertible.
From \eqref{eq:Jacobian} it follows that the matrices $F_b$ and $H_b$ are invertible as well.
Recall that $F_b$ is the Jacobian of the smooth map $\underline{\varphi}: N'\rightarrow N$ at $b$.
Together with the first condition in the lemma, this implies that   $\underline{\varphi}$ is a diffeomorphism.
Meanwhile, the determinant  of $H$ is a nowhere vanishing smooth function on $N$, and hence
  $H$ is invertible.
  Thus, from Example \ref{ex:varphir}, we have that the morphism
  \begin{eqnarray}\label{eq:varphirr+1iso}
\varphi^*_{r,r+1}:\overline\varphi^{-1}(\m_{\CO}^r/\m_{\CO}^{r+1})\rightarrow \m_{\CO'}^r/\m_{\CO'}^{r+1}\quad
(\CO':=\CO_{N'}^{(k)},\ \CO:=\CO_N^{(k)})
\end{eqnarray}
  is an isomorphism for every $r\in \BN$.

By induction on $r$, the commutative diagram
	 	\[\xymatrix{0\ar[r]&\overline\varphi^{-1}(\m_{\CO}^r /\m_{\CO }^{r+1}) \ar[r]\ar[d]^{\varphi^*_{r,r+1}}&
\overline\varphi^{-1}(\CO/\m_{\CO}^{r+1})\ar[r]\ar[d]^{\varphi^*_{r+1}}&\overline\varphi^{-1}(\CO/\m_{\CO}^{r})\ar[r]\ar[d]^{\varphi^*_r}&0
	 	\\0\ar[r]& \m_{\CO'}^r /\m_{\CO' }^{r+1}\ar[r]&\CO'/\m_{\CO'}^{r+1}\ar[r]&\CO' /\m_{\CO'}^{r}\ar[r] & 0}\]
	  and the isomorphism \eqref{eq:varphirr+1iso} imply that
 the  morphism \[\varphi^*_r: \overline\varphi^{-1}(\CO/\m^{r}_{\CO })\rightarrow \CO'/\m^{r}_{\CO'}\]
   of sheaves is an isomorphism.
   Then the lemma follows from the 
 identification \eqref{eq:o=invlim}.
\end{proof}

\vspace{3mm}

\noindent\textbf{Proof of Theorem \ref{thm:invfun}:}
By Lemma \ref{lem:dimtan}, the assumption in Theorem \ref{thm:invfun} implies that \[
\dim_b M'=\dim_{\overline \varphi(b)} M\quad\textrm{and}\quad \deg_b M'=\deg_{\overline \varphi(b)} M.
\] 
Put $n=\dim_b M'$, and $k=\deg_b M'$.	 Choose an open neighborhood  $W'$ of $b$ in $M'$, and an open subset $W$ of $M$ such that
	\[(W', \CO'|_{W'})\cong(\R^n, \CO_{\R^n}^{(k)}),\quad (W, \CO|_{W})\cong(\R^n, \CO_{\R^n}^{(k)})\]
	and $\varphi(W’)\subset W$. By 
 the inverse function theorem for smooth manifolds and Example \ref{ex:Jacobian},
 there is a neighborhood $V'\subset W'$ of $b$ in $M'$ such that
\begin{itemize}
\item $V:=\overline\varphi(V')$ is an open subset of $M'$;
\item $\underline \varphi|_{\underline{V'}}:(V',\underline{\CO'}|_{V'})\rightarrow (V,\underline{\CO}|_V)$ is a diffeomorphism; and
\item $d\varphi_{b'}:\mathrm{T}_{b'}(V')\rightarrow \mathrm{T}_{\overline\varphi(b')}(V)$ is bijective for every $b'\in V'$.
\end{itemize}
 Lemma \ref{lem:invfun} implies that $\varphi|_{V'}$ is an isomorphism, and hence the theorem follows.

\section{Constant rank theorems}\label{sec:rankthm}

In this section, we show
how the local behavior of a morphism of formal manifolds is modeled by the behavior of its rank.

\subsection{Morphisms to $N^{(k)}$} \label{subsec:mortoNk}
 Let  $n,k\in \BN$, and let $N$ be an open submanifold of $\BR^n$. 
 In this subsection we give a characterization of the set of morphisms from $M$ to $N^{(k)}$, which is a generalization of \cite[Theorem 4.9]{CSW1}.

 Write $(x_1,x_2,\dots,x_n,y_1,y_2,\dots,y_k)$ for the standard coordinate system of $N^{(k)}$.
As mentioned in the Introduction, for every morphism 
 \[\varphi=(\overline{\varphi},\varphi^*):\  (M,\CO)\rightarrow (N,\CO_N^{(k)}),\] 
   set 
 \[c_\varphi:=(\varphi^*(x_1),\varphi^*(x_2),\dots,\varphi^*(x_n),\varphi^*(y_1),\varphi^*(y_2),\dots,\varphi^*(y_k))\in (\CO(M))^{n+k}.\] 
It is easy to see that 
\[c_\varphi\in \CO(M;N)\times (\m_\CO(M))^k,\]
where  $\CO(M;N)$ is the subset 
\begin{eqnarray*}
\{(f_1,f_2,\dots,f_n)\in (\CO(M))^n\mid 
(f_1(a),f_2(a),\dots,f_n(a))\in N\ \text{for all}\ a\in M\}
\end{eqnarray*} 
of $(\CO(M))^n$, 
and $\m_\CO(M)$ is the ideal of $\CO(M)$ defined in \eqref{eq:defmo}.

The main goal of this subsection is to prove the following result. 

\begin{prpd}\label{prop:genebij}  The  map 
  \be\label{eq:genebij} \{\text{morphism from $M$ to $N^{(k)}$}\}\rightarrow \CO(M;N)\times (\m_{\CO}(M))^k,\quad \varphi\mapsto c_\varphi\ee is a bijection.  
\end{prpd}

When $N=\BR^n$ and $k=0$, this proposition has been proved in  \cite[Theorem 4.9]{CSW1}. Similar to the proof of
\cite[Lemma 4.10]{CSW1}, 
 the following result is straightforward.
\begin{lemd}\label{lem:mortoformaleq}
Let  $\varphi_1,\varphi_2:  M\rightarrow N^{(k)}$ be two morphisms such that  $c_{\varphi_1}=c_{\varphi_2}$. Then we have that 
    $\varphi_1=\varphi_2$.
\end{lemd}


\begin{lemd}\label{lem:mortoiformal} 
Proposition \ref{prop:genebij} holds when  $N=\R^{0}$ and $k=1$. 
\end{lemd}
\begin{proof}
    By Lemma \ref{lem:mortoformaleq}, it suffices to show that the map \eqref{eq:genebij} is surjective when $N=\R^{0}$ and $k=1$. 
    Let $g\in\m_{\CO}(M)$. 
     Then we have  a morphism $\varphi=(\overline\varphi,\varphi^*): M\rightarrow (\BR^0)^{(1)}$ such that \[ \varphi^*: \ 
        \BC[[y_1]] \rightarrow  
         \CO(M),\qquad 
         \sum_{j\in \BN} c_jy_1^j\mapsto \sum_{j\in \BN} c_jg^j.
    \] 
 Note that the summation $\sum_{j\in \BN} c_jg^j$  is well-defined as  $g\in\m_{\CO}(M)$.
   It is obvious that $c_{\varphi}=g$. Then the lemma follows.
\end{proof}

\noindent\textbf{Proof of Proposition \ref{prop:genebij}:}
Note that \[N^{(k)}=N^{(0)}\times\underbrace{(\BR^0)^{(1)}
\times\dots \times (\BR^0)^{(1)}}_k. \]
Then, by considering Lemma \ref{lem:mortoiformal}, the proposition is reduced to the case when $k=0$. In such case, by Lemma \ref{lem:mortoformaleq}, it suffices to prove that the map \eqref{eq:genebij} is surjective.
   
     Let $(f_1,f_2,\dots, f_n)$ be an $n$-tuple in $\CO(M;N)$. 
    Write $x'_1, x_2',\dots,x_n'$ for the standard coordinate functions of $(\BR^n)^{(0)}$.
    By \cite[Theorem 4.9]{CSW1}, there is a unique morphism $\phi=(\overline\phi, \phi^*): M\rightarrow (\BR^n)^{(0)}$ of formal manifolds such that 
    \[
    (\phi^*(x_1'),\phi^*(x_2'),\dots,\phi^*(x_n'))=(f_1,f_2,\dots, f_n).
    \] Note that for every $a\in M$, we have  that \begin{eqnarray*}
         \overline{\phi}(a)&=&((\phi^*(x'_1))(a),(\phi^*(x'_2))(a),\dots,(\phi^*(x'_n))(a))
    \\&=&(f_1(a),f_2(a),\dots,f_n(a))\in N.   
    \end{eqnarray*}
This gives that $\overline{\phi}(M)\subset N$. Since $N$ is open in $\BR^n$,   there is a natural morphism $\varphi=(\overline\varphi, \varphi^*) :M\rightarrow N^{(0)}$ induced by $\phi$. Moreover, we have  that \[\varphi^*(x_i)=\phi^*(x'_i)=f_i\] for every $i=1,2,\dots,n$, and hence  $c_{\varphi}=(f_1,f_2,\dots,f_n)$. This implies that \eqref{eq:genebij} is surjective when $k=0$, as desired.

\subsection{Constant rank morphisms}\label{subsec:constantrankmor}
In this subsection, we give a version of constant rank theorem for formal manifolds.   
We start with the following definition.

\begin{dfn}\label{de:constantrankmor}Let 
$\varphi=(\overline{\varphi},\varphi^*):(M',\CO')\rightarrow (M,\CO)$ be a morphism of formal manifolds,  $b\in M'$, and  $\bm{r}\in \BN^3$.

\noindent (a) We say that $\varphi$ has constant rank   near $b$ if  there is an open neighborhood $U'$ of $b$ in $M'$  such that $\mathrm{Rank}_{b'}(\varphi)=\mathrm{Rank}_b(\varphi)$  for all $b'\in U'$. In this case, we also say that $\varphi$ has constant rank $\bm{r}$  near $b$ if $\mathrm{Rank}_b(\varphi)=\bm{r}$.

 \noindent   (b) We  say that $\varphi$ has  constant rank $\bm{r}$ if $\mathrm{Rank}_{b'}(\varphi)=\mathbf{r}$ for all $b'\in M'$.

\noindent (c) We say that $\varphi$ has locally constant rank  if for every $b'\in M'$, it has constant rank  near $b'$.

\end{dfn} 

In what follows, we present some examples of constant rank morphisms of formal manifolds.

\begin{exampled}\label{ex:constantrank}
Let $n,k,n',k',r_1,r_2,r_3\in \BN$ with
    \[
n'\ge r_1,\quad n\ge r_1+r_2,\quad k'\ge r_2+r_3,\quad\text{and}\quad k\ge r_3.
\]  
Let $N'\subset \BR^{n'}$ and $N\subset \BR^n$ be open subsets  such that  $p_{r_1}(N')\subset N$, where  
\[
p_{r_1}:\ \R^{n'}\rightarrow \R^n,\quad (a_1,a_2,\dots,a_{n'})\mapsto (a_1,a_2,\dots,a_{r_1},\underbrace{0,\dots,0}_{n-r_1}).
\] Write  $(x_1,x_2,\dots,x_n,y_1,y_2,\dots,y_k)$ and $(u_1,u_2,\dots,u_{n'},z_1,z_2,\dots,z_{k'})$ for the standard coordinate systems of $N^{{(k)}}$ and $(N')^{(k')}$, respectively.
Define  $\varphi: (N')^{(k')}\rightarrow N^{(k)}$ to be the morphism  given by 
\begin{eqnarray}\label{eq:defconsrankmor} 
 && (x_1,\dots,x_{r_1},x_{r_1+1},\dots,x_{n-r_2},x_{n-r_2+1},\dots,x_{n},y_1,\dots,y_{r_3},y_{r_3+1},\dots,y_k)\\  
\notag &\mapsto&
 (u_1,\dots,u_{r_1},g_1,\dots,g_{n-r_1-r_2},z_1,\dots,z_{r_2},z_{r_2+1},\dots,z_{r_2+r_3},h_{1},\dots,h_{k-r_3}),
\end{eqnarray}
where
\begin{eqnarray}
\label{eq:defgi}g_i&=&\sum_{j'=1}^{r_2+r_3}g_{ij'}z_{j'}+\sum_{J\in \BN^{k'}; |J|\ge 2} g_{iJ}z^J\quad (i=1,2,\dots,n-r_1-r_2),\\
\label{eq:defhj}h_j&=&\sum_{j'=r_2+1}^{r_2+r_3}h_{jj'}z_{j'}
+\sum_{J\in \BN^{k'}; |J|\ge 2} h_{jJ}z^J\quad (j=1,2,\dots,k-r_3)
\end{eqnarray}
for some $g_{ij'},g_{iJ},h_{jj'},h_{jJ}\in \RC^\infty(N')$.
It is routine that the morphism $\varphi$ has constant rank 
\[(r_1+r_2+r_3,r_1,r_3).\]
\end{exampled}

The following result shows that every constant rank morphism of formal manifolds is locally modeled by  some morphism defined above.

\begin{thmd}\label{thm:rankthm}
Let $\varphi=(\overline\varphi,\varphi^*):(M',\CO')\rightarrow (M,\CO)$ be a morphism of formal manifolds, and let $b\in M'$.
    Assume that $\varphi$ has constant rank $(r_1+r_2+r_3,r_1,r_3)$ near $b$ for some $r_1,r_2,r_3\in \BN$. 
Then 
 \[
n'\ge r_1,\quad n\ge r_1+r_2,\quad k'\ge r_2+r_3,\quad\text{and}\quad k\ge r_3,
\]
where
\[n':=\dim_{b} M',\quad k':=\deg_{b} M',\quad n:=\dim_{\overline{\varphi}(b)} M,\quad \text{and}\quad k:=\deg_{\overline{\varphi}(b)} M.\]
Furthermore, there exists a chart $(U',(N')^{(k')},\vartheta')$ of $M'$ containing $b$ and a chart $(U,N^{(k)},\vartheta)$ of $M$ containing $\overline{\varphi}(b)$ such that $\overline{\varphi}(U')\subset U$ and the morphism 
\be\label{eq:localrepnvarphi}
 \vartheta\circ   \varphi|_{U'}\circ (\vartheta')^{-1}:   \        (N')^{(k')}\rightarrow N^{(k)}
\ee
is given by  \eqref{eq:defconsrankmor} for some suitable $g_i,h_j$ as in \eqref{eq:defgi}, \eqref{eq:defhj}.

\end{thmd}

\begin{proof} The first assertion is obvious. For the second one, without loss of generality, we  assume that the morphism 
 \[\varphi:\ M'=(N')^{(k')}\rightarrow M=N^{(k)}\]
is  as in Example \ref{ex:Jacobian}, and has constant rank $(r_1+r_2+r_3,r_1,r_3)$.
With the notations as in Example \ref{ex:Jacobian}, by recording the coordinates, we may (and do) assume that the determinants of  the minors 
    \[
 \left((\partial_{u_{i'}}(f_i))(b')\right)_{\substack{1\le i\le r_1,\\1\le i'\le r_1}},\  \left(h_{jj'}(b')\right)_{\substack{1\le j\le r_3,\\r_2+1\le j'\le r_2+r_3}},
 \ \text{and}\ \left(\begin{array}{cc}
   g_{ij'}(b')   \\
    h_{jj'}(b')
 \end{array}\right)_{\substack{n-r_2+1\le i\le n,\\ 1\le j\le r_3,\\
 1\le j'\le r_2+r_3}}
    \]
 in $J_{b'}$   are nonzero for all $b'\in N'$. 
By the constant rank theorem for smooth manifolds (see \cite[Theorem 4.12]{L}), we may (and do) further assume that the reduction 
 $\underline{\varphi}: N'\rightarrow N$ of $\varphi$ is given by 
 \begin{eqnarray}
\label{eq:smoothconsrankthm}  (x_1,x_2,\dots,x_n)\mapsto (u_1,u_2,\dots,u_{r_1},\underbrace{0,\dots,0}_{n-r_1}).   
 \end{eqnarray}
Then it is straightforward to see that 
 \begin{eqnarray}
 \label{eq:gij=0} &&   g_{ij'}=0\quad\text{for}\ i=r_1+1,\dots,n-r_2,\ j'=r_2+r_3+1,\dots,k',\\
\label{eq:hij=0} && h_{jj'}=0\quad\text{for}\ j=r_3+1,\dots,k,\ j'=1,\dots,r_2, r_2+r_3+1,\dots,k'. 
 \end{eqnarray}

By \eqref{eq:smoothconsrankthm}, we have that 
\[
\varphi^*(x_{n-r_2+1}),\varphi^*(x_{n-r_2+2}),\dots,\varphi^*(x_{n})\in \m_{\CO'}(M').
\]
It then follows from Proposition \ref{prop:genebij} that there is a morphism $\psi: (N')^{(k')}\rightarrow (N')^{(k')}$  given by 
\begin{eqnarray*}&&(u_1,\dots,u_{r_1},u_{r_1+1},\dots,u_{n'},z_1,\dots,z_{r_2},z_{r_2+1},\dots,z_{r_2+r_3},z_{r_2+r_3+1},\dots,z_{k'})\\&\mapsto&(\varphi^*(x_1),\dots,\varphi^*(x_{r_1}),u_{r_1+1},\dots,u_{n'},\\&&\varphi^*(x_{n-r_2+1}),\dots,\varphi^*(x_{n}),\varphi^*(y_1),\dots,\varphi^*(y_{r_3}),z_{r_2+r_3+1},\dots,z_{k'}).  \end{eqnarray*}   
Since $d\psi_{b'}$ is bijective for every $b'\in N$,   $\psi$ is an isomorphism by Lemma \ref{lem:invfun}. 
Using \eqref{eq:gij=0} and \eqref{eq:hij=0}, one easily checks that  the morphism $\varphi\circ \psi^{-1}$ has the desired form (which is given by  \eqref{eq:defconsrankmor}).
\end{proof}

\subsection{Proof of Theorem \ref{thm:constantrankthmmain}} \label{subsec:rankthm}
This subsection is devoted to a proof of Theorem \ref{thm:constantrankthmmain}. 
Let $\varphi=(\overline\varphi, \varphi^*): (M',\CO')\rightarrow (M,\CO)$ be a morphism of formal manifolds, and let $b\in M'$ as in Theorem \ref{thm:constantrankthmmain}.

We first give an explicit construction of the morphism given by \eqref{eq:introstrandard}. 

\begin{exampled}\label{ex:closedsub} 
In the setting of Example \ref{ex:constantrank}, assume that $g_i$ and $h_j$ are zero for all $i=1,2,\dots,n-r_1-r_2$ and $j=1,2,\dots,k-r_3$. In this case, write \be\label{eq:defzetar}
\zeta_{\bm{r}}=(\overline{\zeta_{\bm{r}}},\zeta_{\bm{r}}^*):\ (N')^{(k')}\rightarrow N^{(k)}\qquad (\bm{r}:=(r_1+r_2+r_3,r_1,r_3))
\ee for the morphism given by \eqref{eq:defconsrankmor} (see also \eqref{eq:introstrandard}).
Then $\overline{\zeta}_{\bm{r}}=p_{r_1}|_{N'}$ and for every  $f=\sum_{J\in \BN^k} f_J y^J\in \CO_N^{(k)}(U)$,
\be\label{eq:zetafstandard}
(\zeta_{\bm{r}})_{U,U'}^*(f)
=\sum_{\substack{I=(i_1,i_2,\dots,i_{r_2})\in \BN^{r_2},\\J=(j_1,j_2,\dots,j_{r_3})\in \BN^{r_3}}}
(f_{(j_1,j_2,\dots,j_{r_3},\underbrace{0,\dots,0}_{k-r_3})})^{(I)}z^{(I,J)},
\ee
where $U$ is an open subset of $N$, $U'$ is an open subset of $N'$ such that $p_{r_1}(U')\subset U$, 
\[
z^{(I,J)}:=z_1^{i_1}z_2^{i_2}\cdots z_{r_2}^{i_{r_2}}z_{r_2+1}^{j_{1}}z_{r_2+2}^{j_{2}}\cdots 
z_{r_2+r_3}^{j_{r_3}},
\]
and for every $h\in \RC^{\infty}(U)$, 
\[
h^{(I)}:=\left(
\frac{1}{i_1!i_2!\cdots i_{r_2}!}
(\partial_{x_{n-r_2+1}}^{i_1}\partial_{x_{n-r_2+2}}^{i_2}\cdots
\partial_{x_{n}}^{i_{r_2}})(h)
\right)\circ (p_{r_1}|_{U'})\in \RC^\infty(U').
\]
\end{exampled} 

By generalizing the notion of a slice in an open subset of $\BR^n$, we introduce the following definition.  

\begin{dfn}\label{def:defzetar}
Let $n,n',k,k',r\in \BN$ with  \[  
r\le n-n',\quad r\le k',\quad\text{and}\quad k'-r\le k,\] and let $N$ be an open subset of $\BR^n$ containing the origin.  We call $((N')^{(k')},\zeta_{\bm r})$ the  $(n',r,k')$-slice of $N^{(k)}$, where 
\be\label{eq:N'} N'=\{(a_1,a_2,\dots,a_n)\in N\mid a_{i}=0\ 
\text{ for every } i> n'\},\ee $\bm r=(n'+k',n',k'-r)$,
and
\be\label{eq:defzetar*}
\zeta_{\bm r}=\zeta_{(n'+k',n',k'-r)}:\ (N')^{(k')}\rightarrow N^{(k)}\ee  is as in \eqref{eq:defzetar}.
\end{dfn}

By  Borel's lemma,  \eqref{eq:zetafstandard} implies the following result.

\begin{lemd}\label{lem:zetarsur}
Let $\zeta_{\bm r}=(\overline{\zeta_{\bm r}},\zeta_{\bm r}^*):(N')^{(k')}\rightarrow N^{(k)}$ be the morphism as in \eqref{eq:defzetar*}. Then the homomorphism 
\be \label{eq:zetasur} \zeta_{\bm r}^*:\ 
\RC^\infty(N)[[y_1,y_2,\dots,y_k]]\rightarrow \RC^\infty(N')[[z_1,z_2,\dots,z_{k'}]]
 \ee is surjective.
\end{lemd}

The following result is straightforward. 

\begin{lemd}\label{lem:constantrankatob}
Assume that \[\varphi=\zeta_{\bm{r}}: M'=(N')^{(k')}\rightarrow M=N^{(k)}\] 
is as in Example \ref{ex:closedsub}. Then for every $b'\in M'$, the homomorphism $\varphi^*_{b',\mathrm{k}}$ (see \eqref{eq:Introvarphi2}) is surjective. 
\end{lemd}

In view of  Lemma \ref{lem:constantrankatob},  it suffices to prove only the 'if' part of Theorem \ref{thm:constantrankthmmain}.
We start with the following special case. For convenience, when $M'=\{b\}$ contains only one point, we will often not distinguish the spaces $\CO'(M')$, $\CO'_b$ and 
$\widehat{\CO'}_b$.

\begin{lemd}\label{lem:stalkofconstantrank} If 
 $M'=(\BR^0)^{(k')}$, $M=(\BR^0)^{(k)}$  for some $k',k\in \BN$, and $\varphi_{b,\mathrm{k}}^*$ is surjective, then $\varphi$ is standardizable near $b$. \end{lemd}
\begin{proof}  Set $r:=\mathrm{Rank}(d\varphi_b)$.
Note that in this case  \[ \dim \m_{\overline{\varphi}(b)}/\m^2_{\overline{\varphi}(b)}=k\quad\text{and}\quad \dim \m'_b/(\m'_b)^2=k'.\] 
Since $\varphi_{b,\mathrm{k}}^*$ is surjective, 
there exists a subset  $\{y'_{j}\}_{1\leq j\leq k}$ of $\m_{\overline{\varphi}(b)}$ such that 
\begin{itemize}
    \item the images of $\overline{y'_{1}},\overline{y'_{2}}, \dots, \overline{y'_{r}}$ under the map 
    \[ \varphi^*_{b,2}:\ \m_{\overline{\varphi}(b)}/\m^2_{\overline{\varphi}(b)}\rightarrow \m'_b/(\m'_b)^2\]
    are linearly independent;
    \item  the set 
    \[\{\overline{y'_{r+1}},\overline{y'_{r+2}}, \dots, \overline{y'_{k}}\}\] forms a basis of $\ker \varphi^*_{b,2}$; and 
    \item the set
    $\{y'_{r+1},y'_{r+2}, \dots,y'_{k}\}$ is contained in $\ker \varphi^*_{b,0}$ (see \eqref{eq:varphi_0}).
\end{itemize} 
Here for every $f\in \m_{\overline{\varphi}(b)}$, the notation $\overline{f}$ stands for the image of $f$ under the quotient  map \[\m_{\overline{\varphi}(b)}\rightarrow \m_{\overline{\varphi}(b)}/\m^2_{\overline{\varphi}(b)}.\] 
Similar notation will be used for the elements of $\m'_b$.

On the other hand, write $z'_i:=\varphi^*_b(y'_i)$ for $1\leq i\leq r$. Since $\overline{z'_i}=\varphi^*_{b,2}(\overline{y'_i})$ for $1\le i\le r$, it follows that 
  \[\{\overline{z'_1},\overline{z'_2},\dots,\overline{z'_{r}}\}\] is a linearly independent subset of $\m'_b/(\m'_b)^2$.
Extend it to a basis
\[\{\overline{z'_1},\overline{z'_2},\dots,\overline{z'_{r}},\overline{z'_{r+1}},\overline{z'_{r+2}},\dots,\overline{z'_{k'}}\} \]
 of $\m'_b/(\m'_b)^2$, where $z'_{r+1},\dots,z'_{k'}$ are suitable elements in $\m_b'$.

Write $\{y_1,y_2,\dots,y_k\}$ and $\{z_1,z_2,\dots,z_{k'}\}$ for the standard coordinate systems of $M$ and $M'$, respectively. 
Then we have two morphisms 
\[\vartheta: M\rightarrow M\quad \text{and}\quad \vartheta': M'\rightarrow M'\] given by (see Proposition \ref{prop:genebij})
\[
(y_1,y_2,\dots,y_k)\mapsto 
(y'_1,y_2',\dots,y'_k)\quad \text{and}\quad (z_1,z_2,\dots,z_{k'})\mapsto (z'_1,z_2',\dots, z'_{k'}),\]respectively.  By Theorem \ref{thm:invfun}, they are both isomorphisms of formal manifolds.
 It is clear that the morphism  \[\vartheta\circ\varphi\circ(\vartheta')^{-1}: M'\rightarrow M\] has the desired form, which makes $\varphi$ is standardizable near $b$. This completes the proof.
\end{proof}

For the general case, by Theorem \ref{thm:rankthm}, we  assume  without loss of generality  that
\begin{itemize}
    \item $M'=(N')^{(k')}$, where $N'=\mathrm{C}(n')\subset \BR^{n'}$ and $n',k'\in \BN$;
    \item $ M=N^{(k)}$, where $N=\mathrm{C}(n)\subset \BR^{n}$  and $n,k\in \BN$; 
    \item $\varphi$
has the form as in \eqref{eq:defconsrankmor}  with \begin{eqnarray*}
g_i&=&\sum_{j'=1}^{r_2+r_3}g_{ij'}z_{j'}+\sum_{J\in \BN^{k'}; |J|\ge 2} g_{iJ}z^J\quad (i=1,2,\dots,n-r_1-r_2),\\
h_j&=&\sum_{j'=r_2+1}^{r_2+r_3}h_{jj'}z_{j'}
+\sum_{J\in \BN^{k'}; |J|\ge 2} h_{jJ}z^J\quad (j=1,2,\dots,k-r_3)
\end{eqnarray*}
for some $g_{ij'},g_{iJ},h_{jj'},h_{jJ}\in \RC^\infty(N')$; and 
\item $\varphi^*_{b',\mathrm{k}}$ is surjective for each $b'\in M'$.
\end{itemize}
Here, for every $m\in \BN$, write 
\be \label{eq:C(n)}
\mathrm{C}(m):=\{(a_1,a_2,\dots,a_m)\in \BR^m\mid |a_i|<1\ \text{for all}\ i\}.
\ee 

 \begin{lemd}\label{lem:ghinpf2}  
Let $1\leq i\leq n-r_1-r_2$, $1\leq j\leq k-r_3$, and $r_1+1\le l\le n'$. Then

\noindent(a) $g_{iJ}=h_{jJ}=0$ for all $J\in \BN^{k'}\setminus \BN^{r_2+r_3}\times \{0\}^{k'-r_2-r_3}$ with $|J|\ge 2$.

\noindent(b) $\partial_{u_l}(g_{iJ'})=\partial_{u_l}(h_{jJ'})=0$ for all $J'\in \BN^{r_2+r_3}\times \{0\}^{k'-r_2-r_3}$ with $|J'|\ge 2$.

\noindent(c) $\partial_{u_l}(g_{ij'})=\partial_{u_l}(h_{jj''})=0$ for all $1\le j'\le r_2+r_3$ and $r_2+1\le j''\le r_2+r_3$.

 \end{lemd}
\begin{proof} Let $b'=(b'_1,b'_2,\dots,b'_{n'})\in M'$.  For every $s\in \BN$, there is a linear map  \[\varphi^*_{b',s+1}:\ \wh{\m}_{\overline\varphi(b')}/\wh{\m}^{s+1}_{\overline{\varphi}(b')}\rightarrow \wh{\m'}_{b'}/(\wh{\m'}_{b'})^{s+1} \] induced by the local homomorphism $\varphi^*_{b'}:\wh{\CO}_{\overline\varphi(b')}\rightarrow \wh{\CO'}_{b'}$. Recall from   \eqref{eq:reductionvarphi} that there is a  morphism  
 \[\varphi_{b'}:\ M'_{b'}\rightarrow M_{\overline{\varphi}(b')}\] whose induced homomorphism on the spaces of formal functions is exactly the local homomorphism 
$\varphi_{b'}^*:\ \widehat{\CO}_{\overline{\varphi}(b')}\rightarrow \widehat{\CO'}_{b'}$. 
By applying Lemma \ref{lem:stalkofconstantrank} to the morphism $\varphi_{b'}$, we see that 
 \begin{eqnarray}\label{eq:rankvarphbs+1}
   \mathrm{Rank}(\varphi^*_{b',s+1})={{r+s}\choose{s}}-1\quad (r:=r_1+r_2+r_3).  
 \end{eqnarray}

 By (the proof of) \cite[Proposition 2.6]{CSW1}, there is a  $\C$-algebra identification \[\widehat{\CO'}_{b'}=\C[[u_1-b'_1,u_2-b'_2,\dots,u_{n'}-b'_{n'},z_1,z_2,\dots,z_{k'}]].\] Under this identification, 
the canonical homomorphism  $\CO'(M')\rightarrow \widehat{\CO'}_{b'}$ is as follows
\begin{eqnarray}\label{eq:taylorexp}
 \pi:\  \RC^\infty(N')[[z_1,\dots,z_{k'}]]&\rightarrow &\C[[u_1-b'_1,\dots,u_{n'}-b'_{n'},z_1,\dots,z_{k'}]],\\
  \notag \sum_{J'\in \BN^{k'}} f_{J'} z^{J'}&\mapsto & \sum_{J'\in \BN^{k'}}
   \left(\sum_{I'\in \BN^{n'}} f_{I',J'}(b') (u-b')^{I'}\right) z^{J'},
\end{eqnarray}
where $f_{I',J'}(b)$ is the coefficient of 
\[(u-b')^{I'}:=(u_1-b'_1)^{i_1}(u_2-b'_2)^{i_2}\cdots(u_{n'}-b'_{n'})^{i_{n'}}\quad  (I'=(i_1,i_2,\dots,i_{n'})) \] in the Taylor series of $f_{J'}$  at $b'$. From \eqref{eq:defconsrankmor}, we know that 
\[
u_1-b'_1,u_2-b'_2,\dots,u_{r_1}-b'_{r_1},z_1,z_2,\dots,z_{r_2+r_3}\in \varphi^*(\CO(M)).
\]
This together with \eqref{eq:rankvarphbs+1} implies that for every $s\ge 1$, the set 
\[
\{(u-b')^{I_1}z^{J_1}+(\wh{\m'}_{b'})^{s+1}\}_{I_1\in \BN^{r_1}\times\{0\}^{n'-r_1},J_1\in \BN^{r_2+r_3}\times\{0\}^{k'-r_2-r_3}; 0<|I_1|+|J_1|\le s}
\]
forms a basis of the image $\mathrm{im}\,\varphi^*_{b',s+1}$.
Then the assertions (a), (b), and (c) in the lemma follow respectively  by considering the elements 
\begin{itemize}
    \item[(a)] 
$\pi(g_i)+(\wh{\m'}_{b'})^{|J|+1}$ and $\pi(h_j)+(\wh{\m'}_{b'})^{|J|+1}$ in $\mathrm{im}\,\varphi^*_{b',|J|+1}$;
\item[(b)] $\pi(g_i)+(\wh{\m'}_{b'})^{|J'|+2}$ and $\pi(h_j)+(\wh{\m'}_{b'})^{|J'|+2}$ in $\mathrm{im}\,\varphi^*_{b',|J'|+2}$; and
\item[(c)] $\pi(g_i)+(\wh{\m'}_{b'})^{3}$ and $\pi(h_j)+(\wh{\m'}_{b'})^{3}$ in $\mathrm{im}\,\varphi^*_{b',3}$.
\end{itemize}

 \end{proof}

 Set \[N_1:=N\cap (\BR^{r_1}\times \{0\}^{n-r_1-r_2}\times \BR^{r_2})\quad \text{and}\quad N_1':=N'\cap (\BR^{r_1}\times \{0\}^{n'-r_1}).\] 
 Since $N$ and $N'$ are open cubes, there are canonical  projections 
 \[
 p=(\overline{p},p^*):\ N^{(k)}\rightarrow (N_1)^{(r_3)}\quad \text{and}\quad 
  p'=(\overline{p'},(p')^*):\ (N')^{(k')}\rightarrow (N_1')^{(r_2+r_3)}.
 \] given by
 \[(x_1,\dots,x_{r_1},x_{n-r_2+1},\dots x_n,y_1,\dots,y_{r_3} )\mapsto (x_1,\dots,x_{r_1},x_{n-r_2+1},\dots x_n,y_1,\dots,y_{r_3} )\]
 and \[(u_1,\dots,u_{r_1},z_1,\dots ,z_{r_2+r_3} )\mapsto (u_1,\dots,u_{r_1},z_1,\dots ,z_{r_2+r_3} )\]
where, by abuse of notation, \[(x_1,\dots,x_{r_1},x_{n-r_2+1},\dots x_n,y_1,\dots,y_{r_3} )\qaq(u_1,\dots,u_{r_1},z_1,\dots,z_{r_2+r_3} )\] are respectively the standard coordinate system of $(N_1)^{(r_3)}$ and $(N_1')^{(r_2+r_3)}$.
 \begin{lemd}\label{lem:borelforconsrank}
   For each $1\leq i\leq n-r_1-r_2$ and $1\leq j\leq k-r_3$, there exist formal functions 
   $g_i'\ \text{and}\ h_j'$ in $ \CO_{N_1}^{(r_3)}(N_1)=\RC^{\infty}(N_1)[[y_1,\dots,y_{r_3}]]$
   satisfying that \[\varphi^*(p^*(g'_i))=g_i\quad \text{and}\quad \varphi^*(p^*(h_j'))=h_j.\]
     
 \end{lemd}
 \begin{proof}
      By Proposition \ref{prop:genebij},
there is a morphism \[\psi=(\overline{\psi},\psi^*):\  (N_1')^{(r_2+r_3)}\rightarrow N_1^{(r_3)}\] given by 
\begin{eqnarray*}
   && (x_1,\dots,x_{r_1},x_{n-r_2+1},\dots x_n,y_1,\dots,y_{r_3} )\\ &\mapsto & (u_1,\dots,u_{r_1},z_1,\dots,z_{r_2},z_{r_2+1},\dots,z_{r_2+r_3}).
\end{eqnarray*}
Lemma \ref{lem:zetarsur} implies that the homomorphism 
\[
\psi^*:\ \RC^{\infty}(N_1)[[y_1,y_2,\dots,y_{r_3}]]\rightarrow \RC^{\infty}(N_1')[[z_1,z_2,\dots,z_{r_2+r_3}]]
\]
is surjective. On the other hand, 
 Lemma \ref{lem:ghinpf2} implies that 
 \[
 g_i,\ h_j\in (p')^*(\RC^{\infty}(N_1')[[z_1,z_2,\dots,z_{r_2+r_3}]]).
 \]
The lemma then follows from the commutative diagram  
\be \label{diag:projections}\begin{CD}
 (N_1')^{(r_2+r_3)}@>  \psi >>(N_1)^{(r_3)} \\
	@A p' AA          @A  A p A\\
	(N')^{(k')} @> \varphi >>N^{(k)}.  \\
\end{CD} 
\ee

 \end{proof}
\noindent\textbf{Proof of Theorem \ref{thm:constantrankthmmain}} 
With the notations used in (the proof of) Lemma \ref{lem:borelforconsrank}, we write 
 \[x'_{r_1+i}:=x_{r_1+i}-p^*(g'_{i})\quad\text{and}\quad y'_{r_3+j}:=y_{r_3+j}-p^*(h'_{j})\] for  $1\leq i\leq n-r_1-r_2$ and $1\leq j\leq k-r_3$. Let $\vartheta: N^{(k)}\rightarrow N^{(k)} $ be the morphism given by 
\begin{eqnarray*}
 && (x_1,\dots,x_{r_1},x_{r_1+1},\dots,x_{n-r_2},x_{n-r_2+1},\dots,x_{n},y_1,\dots,y_{r_3},y_{r_3+1},\dots,y_k)\\  
\notag &\mapsto&
 (x_1,\dots,x_{r_1},x'_{r_1+1},\dots,x'_{n-r_2}, x_{n-r_2+1},\dots, x_n, y_1,\dots, y_{r_3},y'_{r_3+1},\dots,y'_{k}).
\end{eqnarray*} 
It is easy to see that $d\vartheta$ and $d\underline{\vartheta}$ are both bijective.
By Theorem \ref{thm:invfun}, there exists an open neighborhoods $W$ of $0$ in $N$ such that  $W_1:=\overline{\vartheta}(W)$ is open in $N$ and  the morphism 
\[\vartheta|_{W}: W^{(k)}\rightarrow (W_1)^{(k)}\] is an isomorphism of formal manifolds. Set $W':=(\overline{\varphi})^{-1}(W)$. 
Then the composition morphism 
\[ (W’)^{(k')}\xrightarrow{\varphi|_{W'}} W^{(k)}\xrightarrow{\vartheta}W_1^{(k)}\] has the form as in \eqref{eq:introstrandard}. This completes the proof of Theorem \ref{thm:constantrankthmmain}.

\subsection{Immersions and submersions}\label{subsec:immvssub} 
Throughout this subsection, let 
\[\varphi=(\overline\varphi,\varphi^*):\ (M',\CO')\rightarrow (M,\CO)\] be a morphism of formal manifolds, and let $b\in M'$.  Set 
\[\bm{r}=(r_1+r_2+r_3,r_1,r_3)=\mathrm{Rank}_b
(\varphi)\] for some $r_1,r_2,r_3\in \BN$, and also set  \[n':=\dim_{b} M',\quad k':=\deg_{b} M',\quad n:=\dim_{\overline{\varphi}(b)} M,\quad k:=\deg_{\overline{\varphi}(b)} M.\]

For a chart $(U,N^{(k)},\vartheta)$ of $M$, write $(x_1,\cdots,x_n,y_1,\dots,y_k)$ for the standard coordinate system of $N^{(k)}$. For a chart $(U',(N')^{(k')},\vartheta')$ of $M'$, we also write $(u_1,\cdots,u_{n'},z_1,\dots,z_{k'})$ for the standard coordinate system of $(N')^{(k')}$.
\begin{dfn}
    \noindent (a) We say that $\varphi$ is an immersion at $b$ if the differential  $d\varphi_b$ of $\varphi$ at $b$ is injective. 

  \noindent (b) We say that $\varphi$ is a submersion at $b$ if the differential  $d\varphi_b$ of $\varphi$ at $b$ is surjective.

    \noindent (c) We say that $\varphi$ is an immersion if for every $b'\in M'$, it is an immersion at $b'$. 

  \noindent (d) We say that $\varphi$ is a submersion if for every $b'\in M'$, it is a submersion at $b'$. 
\end{dfn}

The following two examples show that, unlike the case of smooth manifolds,  an immersion or submersion of formal manifolds may not have locally constant rank. 

\begin{exampled}\label{ex:immnotconstant}
    Assume that $M'=\R^{(2)}$,  $M=(\R^2)^{(2)}$, and the morphism  $\varphi$ is given by 
    \[
   (x_1,x_2,y_1,y_2)\mapsto 
    (u_1,f_1z_1,u_1z_1,f_2z_2),
    \]
    where $(x_1,x_2,y_1,y_2)$ and $(u_1,z_1,z_2)$ are respectively the standard coordinate systems of $M$ and $M'$, and  $f_1,f_2$ are two nowhere vanishing smooth functions on $\R$.
Then $\varphi$ is an immersion, but does not have constant rank near $0$. 
\end{exampled}

\begin{exampled}\label{ex:subnotconstant}
    Assume that $M'=(\R^2)^{(2)}$,  $M=(\R^2)^{(1)}$, and the morphism $\varphi$ is given by 
    \[
   (x_1,x_2,y_1)\mapsto 
(u_1,u_1u_2+f_1z_1,f_2z_2),
    \]
    where $(x_1,x_2,y_1)$ and $(u_1,u_2,z_1,z_2)$ are respectively the standard coordinate systems of $M$ and $M'$, and $f_1,f_2$ are two nowhere vanishing smooth functions on $\R^2$.
Then $\varphi$ is a submersion, but does not have constant rank near $(0,c)$ for every $c\in \R$. 
\end{exampled}

As mentioned in the Introduction, we introduce the following definition. 

\begin{dfn}
  We say that $d\varphi_b$ is regular (or $\varphi$ is regular at $b$)  if 
 \[
\mathrm{Rank}(d\varphi_b)=\mathrm{Rank}(d\underline{\varphi}_b)+ \mathrm{Rank}(d\varphi_b^\mathrm{f}).
 \]
\end{dfn}

The following result is straightforward. 

\begin{lemd}\label{lem:regularmor}
The morphism $\varphi$ has constant rank near $b$ if one of the following conditions holds:

  \noindent  (1) $d\varphi_b$ is a bijection;
  
\noindent (2) $\varphi$ is a regular immersion at $b$;

   \noindent (3) $\varphi$ is a regular submersion at $b$.
\end{lemd}

As an immediate consequence of Theorem \ref{thm:rankthm}, we have the following result. 

\begin{prpd} \label{prop:dessubmer}
Assume that  $\varphi$ has constant rank $\bm{r}$ near $b$, and that it is a submersion at $b$. 
Then 
\[r_1+r_2=n\quad\text{and}\quad   r_3=k. \]
Furthermore, there exists a chart $(U',(N')^{(k')},\vartheta')$ of $M'$ containing $b$ and a chart $(U,N^{(k)},\vartheta)$ of $M$ containing $\overline{\varphi}(b)$ such that $\overline{\varphi}(U')\subset U$ and the morphism 
\[
 \vartheta\circ   \varphi|_{U'}\circ (\vartheta')^{-1}:   \       (N')^{(k')}\rightarrow N^{(k)}
\]
is  given by 
\[
(x_1,\dots,x_{r_1},x_{r_1+1},\dots,x_n,y_1,\dots,y_k)\mapsto (u_1,\dots,u_{r_1},z_1,\dots,z_{r_2},z_{r_2+1},\dots,z_{r_2+k}).
\]
\end{prpd}

Proposition \ref{prop:dessubmer}, together with Lemma \ref{lem:regularmor}, implies the following two results. 

\begin{cord}\label{cor:desbijdiff}If the differential $d\varphi_b$ of $\varphi$ at $b$ is a bijection, then \[r_1=n', \quad  r_2=n-n'=k'-k, \qaq r_3=k.\] Furthermore, there exists a chart $(U',(N')^{(k')},\vartheta')$ of $M'$ containing $b$ and a chart $(U,N^{(k)},\vartheta)$ of $M$ containing $\overline{\varphi}(b)$ such that $\overline{\varphi}(U')\subset U$ and the morphism 
\[
 \vartheta\circ   \varphi|_{U'}\circ (\vartheta')^{-1}: \ (N')^{(k')}\rightarrow N^{(k)}
\]
is  given by 
\[
(x_1,\dots,x_{n'},x_{n'+1},\dots,x_n,y_1,\dots,y_k)\mapsto (u_1,\dots,u_{n'},z_1,\dots,z_{k'-k},z_{k'-k+1},\dots,z_{k'}).
\]
\end{cord}

\begin{cord}\label{cor:desregsumer}
If  $\varphi$ is a regular submersion at $b$, then 
there exists a chart $(U',(N')^{(k')},\vartheta')$ of $M'$ containing $b$ and a chart $(U,N^{(k)},\vartheta)$ of $M$ containing $\overline{\varphi}(b)$ such that $\overline{\varphi}(U')\subset U$ and the morphism 
\[
 \vartheta\circ   \varphi|_{U'}\circ (\vartheta')^{-1}:          (N')^{(k')}\rightarrow N^{(k)}  
\] is given by
\begin{eqnarray*}
    (x_1,\dots,x_{n},y_1,\dots,y_k)\mapsto (u_1,\dots,u_{n},z_1,\dots,z_{k}).
\end{eqnarray*}

\end{cord}

As usual, by a local section of $\varphi$, we mean a morphism \[\psi=(\overline{\psi},\psi^*):\ (U,\CO|_U)\rightarrow (M',\CO')\quad (U\ \text{is an open subset of}\ M)\] of formal manifolds such that $\varphi\circ \psi=\mathrm{Id}_U$ in the sense that 
\begin{itemize}
\item $(\overline{\varphi}\circ \overline{\psi})(a)=a$ for all $a\in U$, and 
\item for every open subset $V$ of $U$, the composition 
\[
\psi_{V',V}^*\circ\varphi^*_{V,V'}:\ \CO(V)\rightarrow \CO'(V')\rightarrow \CO(V)\quad (V':=\overline{\varphi}^{-1}(V))
\]
is the identity homomorphism on $\CO(V)$.
\end{itemize} 
The following result generalizes the local section theorem for smooth manifolds (see \cite[Theorem 4.26]{L}).

\begin{prpd}
    The morphism $\varphi$ is a regular submersion at $b$ if and only if there exists a local section 
    \[\psi=(\overline{\psi},\psi^*):\ (U,\CO|_U)\rightarrow (M',\CO')\]
    of $\varphi$ such that $b$ is contained in the image of $\overline{\psi}$.
\end{prpd}
\begin{proof}
    Assume that $\varphi$ is a regular submersion at $b$. With the notations as in  Corollary \ref{cor:desregsumer}, and assume further that $N=\mathrm{C}(n)$ and $N'=\mathrm{C}(n')$ (see \eqref{eq:C(n)}). By Proposition \ref{prop:genebij},  there is a morphism 
    \[ \phi:\ N^{(k)}\rightarrow (N')^{(k')}\] 
    given by 
     \begin{eqnarray*}
         &&(u_1,\dots,u_{n},u_{n+1},\dots,u_{n'},z_1,\dots,z_k,z_{k+1},\dots,z_{k'})\\
         &\mapsto&(x_1,\dots,x_n,\underbrace{0,\dots,0}_{n'-n},y_1,\dots,y_k,\underbrace{0,\dots,0}_{k'-k}).
     \end{eqnarray*}
It is clear that the composition map 
\[\psi: \ (U,\CO|_U)\xrightarrow{\vartheta} N^{(k)}\xrightarrow{\phi} (N')^{(k')}\xrightarrow{(\vartheta')^{-1}} (U',\CO'|_{U'})\rightarrow (M',\CO')\] is a desired local section of $\varphi$.

Conversely, if there exists a local section  \[\psi=(\overline{\psi},\psi^*):\ (U,\CO|_U)\rightarrow (M',\CO')\] of $\varphi$ such that $b$ lies in the image of $\overline{\psi}$, then    \[d\varphi_{b}\circ d\psi_a=\mathrm{id}_{\mathrm{T}_a(U)}\quad \text{and}\quad  d\underline\varphi_{b}\circ d\underline\psi_a=\mathrm{id}_{\mathrm{T}_a(\underline U)}\quad (a\in U\ \text{with}\ \overline{\psi}(a)=b).
\]  
This implies that $\varphi$ is a regular submersion at $b$.
\end{proof}

By applying  Theorem \ref{thm:constantrankthmmain}, we have the following result. 

\begin{prpd}\label{prop:localimmersion}
    Assume that the morphism $\varphi:M'\rightarrow M$ has constant rank $\bm{r}$ near $b$, and that it is an immersion at $b$. Then \[r_1=n' \quad \text{and} \quad r_2+r_3=k'.\] Furthermore,
there exists a chart $(U',(N')^{(k')},\vartheta')$ of $M'$ containing $b$ and a chart $(U,N^{(k)},\vartheta)$ of $M$ containing $\overline{\varphi}(b)$ such that $\overline{\varphi}(U')\subset U$ and the morphism 
\[
 \vartheta\circ   \varphi|_{U'}\circ (\vartheta')^{-1}: \         (N')^{(k')}\rightarrow N^{(k)}  
\] is given by 
\begin{eqnarray*}
&& (x_1,\dots,x_{n'},x_{n'+1},\dots ,x_{n-r_2},x_{n-r_2+1},\dots,x_n,y_1,\dots,y_{r_3},y_{r_3+1},\dots,y_k)\\
&\mapsto& (u_1,\dots,u_{n'},\underbrace{0,\dots,0}_{n-n'-r_2},z_1,\dots,z_{r_2},z_{r_2+1},\dots,z_{k'},\underbrace{0,\dots,0}_{k-r_3}).
\end{eqnarray*}
     \end{prpd}
\begin{proof} 
The first assertion is clear. For the second one, we claim that the map 
\[\varphi_{b',\mathrm{k}}^*:\ \ker \varphi^*_{b',1}\rightarrow \ker \varphi^*_{b',2}\] is surjective if $\varphi$ has constant rank $\bm{r}$ near $b'\in M'$ and it is an immersion at $b'$.
Note that, according to Theorem \ref{thm:constantrankthmmain}, the section assertion is implied by this claim. 

In what follows, we prove the claim.
By Theorem \ref{thm:rankthm}, we may (and do) assume  $\varphi$ is as in Example \ref{ex:constantrank}, $b'=0\in N'$ and $\overline{\varphi}(b')=0\in N$. Recall that we have the $\C$-algebra identifications 
\[
\widehat{\CO}_{\overline{\varphi}(b')}=\C[[x_1,\dots,x_n,y_1,\dots,y_k]]\ \text{and}\ 
\widehat{\CO'}_{b'}=\C[[u_1,\dots,u_{n'},z_1,\dots,z_{k'}]],
\]
and that  the canonical homomorphisms $\CO(M)\rightarrow \widehat{\CO}_{\overline{\varphi}(b')}$ and 
$\CO'(M')\rightarrow \widehat{\CO'}_{b'}$ are given as in \eqref{eq:taylorexp} (with $b'=0$). Consider the morphism\[ \varphi^*_{b'}:\ \wh{\CO}_{\overline{\varphi}(b')}\rightarrow \wh{\CO'}_{b'}\] between formal stalks.
One concludes from \eqref{eq:defconsrankmor}  
that 
\[
\varphi_{b'}^*(x_i)=u_i,\quad \varphi_{b'}^*(x_{n-r_2+i'})=z_{i'},\quad \varphi_{b'}^*(y_j)=z_{j+r_2}
\]
for $1\le i\le n'$, $1\le i'\le r_2$ and $1\le j\le r_3$.
Thus, the restriction of 
$ \varphi^*_{b'}: \wh{\CO}_{\overline{\varphi}(b')}\rightarrow \wh{\CO'}_{b'}$ 
on 
\[A:=\C[[x_1,\dots,x_{n'},x_{n-r_2+1},\dots,x_{n},y_{1},\dots,y_{r_3}]]
\] is an isomorphism.

For each $1\leq i\leq n-n'-r_2$ and $1\leq j\leq k-r_3$, take  $g'_{i}$ and $h'_{j}$ in $A$ such that 
$\varphi_{b'}^*(g'_i)$ and $\varphi_{b'}^*(h_j')$ are respectively the image of $g_i$ and $h_j$ (see Example \ref{ex:constantrank}) under the map $\CO'(M')\rightarrow \wh{\CO'}_{b'}$.
 Then \[\{g'_{i}-x_{n'+i}\}_{1\leq i\leq n-n'-r_2}\cup \{ h'_{j}-y_{k'+j}\}_{1\leq j\leq k-r_3} \] is a subset of $\ker \varphi^*_{b',0}$, and its image 
 \[\{(g'_{i}-x_{n'+i})+(\m_{b'}')^2\}_{1\leq i\leq n-n'-r_2}\cup \{ (h'_{j}-y_{k'+j})+(\m_{b'}')^2  \}_{1\leq j\leq k-r_3}\]
 under the map $\varphi^*_{b',\mathrm{k}}$ is a linearly independent subset in $\ker \varphi^*_{b',2}$. On the other hand, since $\varphi$ is an immersion at $b'$, it follows from Remark \ref{rem:dvarphivsvarphi2} that 
 \[\mathrm{{dim}} \ker \varphi^*_{b,2}=n+k-n'-k'= n+k-n'-r_2-r_3.\] This implies that the map $\varphi^*_{b',\mathrm{k}}$
is surjective, as required. 


\end{proof}

Proposition \ref{prop:localimmersion}, together with Lemma \ref{lem:regularmor}, implies the following result.

\begin{cord}\label{cor:desregimmer}
If  $\varphi$ is a regular immersion at $b$, then 
there exists a chart $(U',(N')^{(k')},\vartheta')$ of $M'$ containing $b$ and a chart $(U,N^{(k)},\vartheta)$ of $M$ containing $\overline{\varphi}(b)$ such that $\overline{\varphi}(U')\subset U$ and the morphism 
\[
 \vartheta\circ   \varphi|_{U'}\circ (\vartheta')^{-1}:          (N')^{(k')}\rightarrow N^{(k)}  
\] is given by
\begin{eqnarray*}
&& (x_1,\dots,x_{n'},x_{n'+1},\dots,x_n,y_1,\dots,y_{k'},y_{k'+1},\dots,y_k)\\
&\mapsto& (u_1,\dots,u_{n'},\underbrace{0,\dots,0}_{n-n'},z_1,\dots,z_{k'},\underbrace{0,\dots,0}_{k-k'}).
\end{eqnarray*}
\end{cord}

\section{Formal submanifolds}\label{sec:submanifold}

In this section, we study the theory of formal submanifolds. 

\subsection{Proof of Theorem \ref{thm:charclosedsub}} \label{subsec:charcloembd}
This subsection is devoted to a proof of Theorem \ref{thm:charclosedsub}. 

 Let $n, n', r,k,k'\in \BN$ with $n-n'=r$, and let $N$ be an open subset of $\BR^n$ containing the origin. 
Let $((N')^{(r+k)},\zeta_{\bm r})$ denote the $(n',r,r+k)$-slice  of $N^{(k)}$ (see Definition \ref{def:defzetar}). Recall that $N'=N\cap (\BR^{n'}\times\{0\}^r)$, and 
\[\zeta_{\bm r}:\ (N')^{(r+k)}\rightarrow N^{(k)}\]
is the morphism  given by 
  \[
(x_1,\dots,x_{n'},x_{n'+1},\dots,x_n,y_1,\dots,y_k)\mapsto (u_1,\dots,u_{n'},z_1,\dots,z_r,z_{r+1},\dots,z_{r+k}),
  \]
where  $(x_1,\dots,x_n,y_1,\dots,y_k)$ and  $(u_1,\dots,u_{n'},z_1,\dots,z_{r+k})$ are respectively the standard coordinate systems of $N^{(k)}$ and $(N')^{(r+k)}$. 
\begin{lemd}\label{lem:constructphi}
Let $\psi=(\overline\psi,\psi^*): M\rightarrow N^{(k)}$ be a morphism of formal manifolds such that $
 \overline{\psi}(M)\subset N'$. 
  Then there is a unique  morphism $\phi: M\rightarrow (N')^{(r+k)}$  such that the diagram
  \[\xymatrix{ &(N')^{(r+k)}\ar[d]^{\zeta_{\bm r}}\\
  M\ar[r]^{\psi}\ar[ur]^{\phi}&N^{(k)}}
  \]  commutes. 
\end{lemd}
\begin{proof} 
Let $\phi=(\overline{\phi},\phi^*): M \rightarrow (N')^{(r+k)} $ be a morphism of formal manifolds. Note that  the composition $\zeta_{\bm r}\circ \phi: M\rightarrow N^{(k)}$ is given by 
\begin{eqnarray*}
&&(x_1,\dots,x_{n'},x_{n'+1},\dots,x_{n},y_{1},\dots,y_k)\\
&\mapsto&  
(\phi^*(u_1),\dots,\phi^*(u_{n'}),\phi^*(z_{1}),\dots,\phi^
*(z_r),\phi^*(z_{r+1}),\dots,\phi^*(z_{r+k})).    
\end{eqnarray*}
This together with Lemma \ref{lem:mortoformaleq} implies that $\psi=\zeta_{\bm r}\circ \phi$ if and only if the morphism $\phi$ is given by 
\[
(u_1,\dots,u_{n'},z_1,\dots,z_r,z_{r+1},\dots,z_{r+k})\mapsto \bm{f},
\]
where 
\[\bm{f}:=(\psi^*(x_1),\dots,\psi^*(x_{n'}),
\psi^*(x_{n'+1}),\dots,\psi^*(x_{n}),
\psi^*(y_1),\dots,\psi^*(y_k)).
\]

 For any $a\in M$,  the condition $\overline{\psi}(M)\subset N'$ implies that \[(\psi^*(x_1),\dots,\psi^*(x_{n'}))\in \CO(M;N')\] and that $(\psi^*(x_{n'+i}))(a)=0$ for all $i=1,2,\dots,r$. 
Therefore
\[\bm{f}\in \CO(M;N')\times (\m_{\CO}(M))^{r+k}.\] Then the lemma follows by  Proposition \ref{prop:genebij}. 
\end{proof}

In the rest of this subsection, 
 let \[\varphi=(\overline\varphi,\varphi^*): \ (M',\CO')\rightarrow (M,\CO)\] be a morphism of formal manifolds as in Theorem \ref{thm:charclosedsub}. 

\begin{lemd}\label{lem:sur}
    Assume that there is  an open cover $\{U_{\gamma}\}_{\gamma\in \Gamma}$ of $M$  such that the homomorphism \[\varphi_{U_\gamma}^*:\ \CO(U_{\gamma})\rightarrow \CO'(\overline{\varphi}^{-1}(U_{\gamma}))\] is surjective for every $\gamma\in \Gamma$. Then the homomorphism \[\varphi^*:\ \CO(M)\rightarrow\CO'(M')\] is surjective as well.
\end{lemd}
\begin{proof}  Let $f\in \CO'(M')$. 
 For each $\gamma\in \Gamma$, we choose a formal function $h_{\gamma}\in \CO(U_{\gamma})$ such that $\varphi^*_{U_{\gamma}}(h_{\gamma})=f|_{\overline\varphi^{-1}(U_{\gamma})}$.
Let $\{g_\gamma\}_{\gamma\in \Gamma}$ be a partition of unity on $M$ subordinate to the open cover  $\{U_{\gamma}\}_{\gamma\in \Gamma}$ (see \cite[Proposition 2.3]{CSW1}).  
Then it is straightforward that 
\[\varphi^*\left(\sum_{\gamma\in\Gamma}
g_{\gamma}h_{\gamma}\right)=f.\]
This implies the lemma.
   \end{proof}

The following result follows easily from Lemma \ref{lem:taninj}, \eqref{eq:chainrule}, and \eqref{eq:underlinecommdiag}.

\begin{lemd}\label{lem:infinj}
Suppose  $\varphi$ is an immersion at a point $b\in M'$. 
 Then $\underline\varphi: \underline{M'}\rightarrow \underline{M}$ and $\varphi_b: M'_b \rightarrow M_{\overline{\varphi}(b)}$  are immersions at $b$ as well.
\end{lemd}

Recall from \cite[Definition 5.6]{CSW1} that 
a formal $\BC$-algebra is a topological $\BC$-algebra that is topologically isomorphic to $\CO(M)$ for some formal manifold $(M,\CO)$.
There is an obvious  contravariant functor
\begin{eqnarray}\label{eq:functor1}
   (M,\CO)\mapsto \CO(M)
\end{eqnarray}
from the category of formal manifolds to the category of formal $\C$-algebras. 
Conversely, it is proved in \cite[Theorem 5.11]{CSW1} that there is a contravariant functor 
\begin{eqnarray}\label{eq:functor2}
  A\mapsto (\mathrm{Specf}(A),\CO_A)
\end{eqnarray}
from the category of  formal $\C$-algebras to the category of  formal manifolds  such that \eqref{eq:functor1} and \eqref{eq:functor2} are quasi-inverse of each other.

 \begin{lemd}\label{lem:submanifoldpre1} 
 Suppose  $M'=N^{(k')}$, $M=N^{(k)}$,  
 $\varphi$ is   an immersion and its reduction $\underline{\varphi}:N\rightarrow N$ is the identity map. 
 Then the  homomorphism 
\[\varphi^*:\ \CO(M)\rightarrow\CO'(M')\]
 is surjective. 
\end{lemd}
\begin{proof}  
Denote by $(x_1,x_2,\dots,x_{n},z_1,z_2,\dots,z_{k'})$  the standard coordinate systems of $N^{(k')}$.
For $i=1,2,\dots,n$ and $j=1,2,\dots,k$, write 
   \begin{eqnarray*}
\varphi^*(x_i)&=&x_i+\sum_{j'=1}^{k'} g_{ij'} z_{j'}+\sum_{J\in \BN^{k'}, |J|\ge 2} g_{iJ} z^J,\\
\varphi^*(y_j)&=&\sum_{j'=1}^{k'} h_{jj'} z_{j'}+\sum_{J\in \BN^{k'}, |J|\ge 2} h_{jJ} z^J,
              \end{eqnarray*}
              where $g_{ij'}, g_{iJ}, h_{jj'}, h_{jJ}\in \RC^\infty(N)$. 
 Since $\varphi$ is an immersion, one concludes from Example \ref{ex:tanvec} that  for every $a\in N$, the Jacobian   
\[
J_a=\left(
  \begin{array}{cc}
    F_a & G_a \\
    0 & H_a \\
  \end{array}
\right)
\]
of $\varphi$ at $a$  has rank $n+k'$, where $F_a$ is the identity matrix of size  $n$,
\[ 
G_a=\left(g_{ij'}(a)\right)_{\substack{1\le i\le n,\,1\le j'\le k'}}\quad \text{and}\quad  H_a=\left(h_{jj'}(a)\right)_{\substack{1\le j\le k,\,1\le j'\le k'}}.\]

 Choose a subset $\{\nu_1,\nu_2,\dots,\nu_{k'}\}$ 
    of $\{1,2,\dots,k\}$ such that the submatrix
    \be\label{eq:Ja}
   \left(
  \begin{array}{cc}
    F_a & G_a \\
    0 & H_a' \\
  \end{array}
\right) \ee 
   of $J_a$ is non-singular, where \[H_a':=\left(h_{\nu_{j'} j'}(a)\right)_{1\le j'\le k'}.\]
Note that  the restriction of  \[\varphi^*:\ \CO(M)=\RC^{\infty}(N)[[y_1,y_2,\dots,y_k]]\rightarrow\CO'(M')= \RC^{\infty}(N')[[z_1,z_2,\dots,z_{k'}]]\] on 
 $A:= \RC^\infty(N)[[y_{\nu_1},y_{\nu_2},\dots,y_{\nu_{k'}}]]$
induces a morphism 
\[\psi=(\overline{\psi},\psi^*):\ M'=\mathrm{Specf}(\CO'(M')) \rightarrow \mathrm{Specf}(A)\]
through the functor \eqref{eq:functor2}. Precisely, $\overline{\psi}: N\rightarrow N$ is the identity morphism, and $\psi^*_U$ is the map 
\[\CO_A(U)=\RC^{\infty}(U)[[y_{\nu_1},y_{\nu_2},\dots,y_{\nu_{k'}}]]\xrightarrow {\varphi^*_{U}|_{\RC^{\infty}(U)[[y_{\nu_1},y_{\nu_2},\dots,y_{\nu_{k'}}]]}}\CO'(U)\] for every open subset $U$ of $N$.
Since the Jacobian of $\psi$ at $a$ equals \eqref{eq:Ja}, it follows from 
  Theorem \ref{thm:invfun} that there is an open neighborhood $U_a$ of $a$ in $N$ such that 
   the restriction of 
   \[
   \varphi_{U_a}^*:\ \CO(U_a)=
   \RC^\infty(U_a)[[y_1,y_2,\dots,y_k]]
   \rightarrow  \CO'(U_a)=\RC^\infty(U_a)[[z_1,z_2,\dots,z_{k'}]]
   \] on $ \RC^\infty(U_a)[[y_{\nu_1},y_{\nu_2},\dots,y_{\nu_{k'}}]]$  is an isomorphism.
   Thus, we obtain  an open cover $\{U_a\}_{a\in N}$ of $N$ such that  the homomorphism  \[
   \varphi_{U_a}^*:\  \CO(U_a)
   \rightarrow \CO'(U_a)\]
  is surjective for every $a\in N$. 
  The lemma then follows from Lemma \ref{lem:sur}.
\end{proof}

\begin{lemd}\label{lem:homsurkey}  Suppose  $M'=(N')^{(k')}$, $M=N^{(k)}$,  
 $\varphi$ is   an immersion and its reduction $\underline{\varphi}: N'\rightarrow N$ is the natural inclusion. 
  Then  the homomorphism 
  \[\varphi^*:\ \CO(M)\rightarrow\CO'(M')\] 
  is surjective.
\end{lemd}
\begin{proof}
    By Lemma \ref{lem:constructphi}, there is a unique morphism $\phi: (N')^{(k')}\rightarrow (N')^{(r+k)}$ such that 
$\zeta_{\bm r}\circ \phi= \varphi$.
 Easy to check that $\phi$ is an immersion and its reduction is the identity map. 
The lemma then follows from Lemmas \ref{lem:submanifoldpre1}  and \ref{lem:zetarsur}.    
\end{proof}
\begin{lemd}\label{eq:surofstalk}
     Let $b\in M'$. If  the  homomorphism $\varphi^*_b: \CO_{\overline\varphi(b)}\rightarrow \CO'_{b}$ is surjective, then  the morphism $\varphi$ is an immersion at $b$.
\begin{proof}
    The lemma is clear by the definition of $\varphi^*_b$ (see \eqref{eq:diffmapata}). 
\end{proof}
\end{lemd}
\vspace{3mm}
\noindent\textbf{Proof of Theorem \ref{thm:charclosedsub} (b):}
Assume first that the homomorphism 
  \[\varphi^*:\ \CO(M)\rightarrow\CO'(M')\] 
  is subjective. This implies that 
\begin{itemize}
 \item[(a)]
    the map (see \cite[(5.5) and Theorem 5.7]{CSW1})
\[
\overline\varphi:\ M'=\mathrm{Specf}(\CO'(M'))
\rightarrow M=\mathrm{Specf}(\CO(M))
\] is injective; 
\item[(b)]  for every $b\in M'$, the homomorphism 
  \[\varphi^*_b:\ \CO_{\overline\varphi(b)}\rightarrow \CO'_{b}\]
is surjective by \cite[Corollary 2.12]{CSW1}; and 
\item[(c)] the homomorphism 
\[\underline{\CO}(M)\rightarrow \underline{\CO'}(M')
\] is  surjective by \cite[Proposition 2.14]{CSW1} and \eqref{eq:underlinecommdiag}.
\end{itemize} 
The facts (a),(b), and Lemma \ref{eq:surofstalk} imply  that $\varphi:M'\rightarrow M$ 
is an injective immersion.
Then $\underline \varphi: \underline {M'}\rightarrow \underline{M}$ is an injective immersion as well by Lemma \ref{lem:infinj}. This, along with the fact (c) and \cite[Chapter 1, Exercise 11]{War}, implies that $\overline\varphi$ is a topological embedding and $\overline\varphi(M')$ is closed in 
$M$. Therefore, $\varphi$ is a closed embedded immersion. 

Conversely, assume that $\varphi:M'\rightarrow M$ 
is a closed embedded immersion. 
For every $b\in M'$, set 
\[
n'_b:=\dim_b M',\ k'_b:=\deg_b M',\ 
n_b:=\dim_{\overline\varphi(b)} M,\ \text{and}\ 
k_b:=\deg_{\overline\varphi(b)} M.
\]
Since the map $\underline{\varphi}:\underline{M'}\rightarrow \underline M$ is a closed embedded immersion, it is easy to see that there exists  an open neighborhood $U_b$ of $\overline\varphi(b)$ in $M$ and an open cube $W_b$ in $\R^{n_b}$ such that 
\[(U_b,\CO|_{U_b})\cong W_b^{(k_b)},\quad  (\overline\varphi^{-1}(U_b),\CO'|_{\overline\varphi^{-1}(U_b)})\cong(W'_b)^{(k'_b)},\]
and the reduction of the morphism 
$(W'_b)^{(k'_b)}\rightarrow W_b^{(k_b)}$ induced by $\varphi|_{\overline\varphi^{-1}(W_b)}$ is the natural inclusion, where
\[W'_b:=(\BR^{n'_b}\times \{0\}^{n_b-n'_b})\cap W_b.
\] Then the map 
\[\varphi^*_{U_b}: \CO(U_b)\rightarrow \CO'(\overline\varphi^{-1}(U_b))\] is surjective by Lemma \ref{lem:homsurkey}. Set $U_0:=M\backslash \overline{\varphi}(M')$. Using the  open cover $\{U_b\}_{b\in M'\bigsqcup\{0\}}$ of $M$, 
 the map $\varphi^*: \CO(M)\rightarrow \CO'(M')$ is surjective by  Lemma \ref{lem:sur}. This completes the proof.
\qed
\vspace{3mm}

By applying Theorem \ref{thm:charclosedsub} (b), we have the following result, which implies Theorem \ref{thm:charclosedsub} (a).

\begin{prpd}\label{prpd:immersion} Let  $\varphi=(\overline{\varphi},\varphi^*): M'\rightarrow M$ be a morphism of formal manifolds, and let $b\in M'$. Then the following conditions are equivalent to each other.
\begin{itemize}
    \item[(a)] The morphism $\varphi$ is an immersion at $b$.
    \item[(b)] The  homomorphism $\varphi^*_b: \CO_{\overline\varphi(b)}\rightarrow \CO'_{b}$ is surjective.
    \item[(c)] The  homomorphism $\varphi^*_b: \wh{\CO}_{\overline\varphi(b)}\rightarrow \wh{\CO'}_{b}$ is surjective.
\end{itemize}
\end{prpd}
\begin{proof} 
``(a) $\Rightarrow$ (b)".
If  $d\varphi_b$ is injective, then so is $d\underline{\varphi}_b$ (see Lemma \ref{lem:infinj}). By \cite[Proposition 1.35]{War} there is 
 an open neighborhood $U'$ of $b$ in $M'$, an open neighborhood $U$ of $\overline{\varphi}(b)$ in $M$ with $\overline{\varphi}(U')\subset U$, an open cube $W$ in $\BR^{n_b}$ ($n_b:=\dim_{\overline{\varphi}(a)}M$), and a slice $W'$ of $W$ such that
\begin{itemize}
    \item $U\cong W$, $U'\cong W'$ as smooth manifolds, and 
    \item the morphism $W'\rightarrow W$
induced by $\overline{\varphi}|_{U'}$ is the natural inclusion.
\end{itemize}
Then the map \[\varphi|_{U'}: \ (U',\CO'|_{U'})\rightarrow (U,\CO|_{U})\] is a closed embedded immersion of formal manifolds.
This, together with Theorem \ref{thm:charclosedsub} (b), implies that the homomorphism  \[\CO(U)\rightarrow \CO'(U')\] is subjective. Then the homomorphism  \[\varphi^*_b: \CO_{\overline{\varphi}(b)}\rightarrow \CO'_b\] is surjective
by \cite[Corollary 2.12]{CSW1}.

``(b) $\Rightarrow$ (c)". This is implied by  the fact that the canonical map $\CO'_b\rightarrow \widehat{\CO'}_b$ is surjective (see \cite[Proposition 2.6]{CSW1}). 

``(c) $\Rightarrow$ (a)".
If the homomorphism $\varphi^*_b: \wh{\CO}_{\overline\varphi(b)}\rightarrow \wh{\CO'}_{b}$ is surjective, then its continuous transpose (see \eqref{eq:tvarphia})
\[
{}^t\varphi_b^*:\ \mathrm{Dist}_b(M')\rightarrow 
\mathrm{Dist}_{\overline\varphi(b)}(M)
\]
is injective. 
This, together with \eqref{eq:diffmapasres}, implies that  
$d\varphi_b$ is injective.
\end{proof}

\subsection{Immersed formal submanifolds}

In the category of smooth manifolds, a monomorphism is nothing but an injective smooth map. In the case of formal manifolds, an injective morphism may not be a monomorphism.  However, we still have the following result by applying Proposition \ref{prpd:immersion}.

\begin{prpd}\label{prop:sub=mono}
Let  \[\varphi=(\overline{\varphi},\varphi^*):\ (M',\CO')\rightarrow (M,\CO)\] be an injective immersion of formal manifolds. Then it is a monomorphism in the category of formal manifolds. 
\end{prpd}
\begin{proof}
 Let
\[\psi_1=(\overline{\psi_1},\psi_1^*),\,\psi_2=(\overline{\psi_2},\psi_2^*):\ (M_1,\CO_1)\rightarrow (M',\CO')\]
be  two morphisms of formal manifolds such that 
\[\varphi\circ\psi_1 =\varphi\circ \psi_2:\ M_1\rightarrow M' \rightarrow M,\] where 
 $(M_1,\CO_1)$ is another formal manifold.
We aim to prove that $\psi_1=\psi_2$.

Note that the injectivity of $\overline\varphi$ implies that
$\overline{\psi_1}=\overline{\psi_2}$. 
Then it suffices to prove that 
for every $a\in M_1$, 
\[
\psi_{1,a}^*=\psi_{2,a}^*:\
\CO'_b\rightarrow \CO_{1,a} \quad (b:=\overline{\psi_1}(a)=\overline{\psi_2}(a)).
\]
Since $\varphi$ is an immersion, it follows from 
Proposition \ref{prpd:immersion} that the homomorphism 
\[
\varphi_b^*:\ \CO_{\overline\varphi(b)}
\rightarrow \CO'_{b}
\]
is surjective. 
This, together with the fact that
\[ 
\psi_{1,a}^*\circ \varphi_b^* =
(\varphi\circ\psi_1 )^*_a=(\varphi\circ \psi_2)^*_a
= \psi_{2,a}^*\circ\varphi_b^*:\ 
\CO_{\overline\varphi(b)}
\rightarrow \CO'_{b}\rightarrow \CO_{1,a},
\]
implies that $\psi_{1,a}^*=\psi_{2,a}^*$ for every $a\in M_1$. The proposition then follows. 
\end{proof}




Recall the family $\mathrm{Imm}(M)$ defined in the Introduction. For two pairs  $(M_1,\varphi_1)$ and $(M_2,\varphi_2)$ in $\mathrm{Imm}(M)$, 
we write \be \label{eq:preorder}(M_1,\varphi_1) \preccurlyeq (M_2,\varphi_2)\ee if there is a morphism $\psi:M_1\rightarrow M_2$ such that $\varphi_2\circ \psi=\varphi_1$. Note that ``$\preccurlyeq$" is a preorder on $\mathrm{Imm}(M)$. 
Proposition \ref{prop:sub=mono} implies that such a morphism $\psi$ is unique (if it exists). Then there is an equivalence relation on $\mathrm{Imm}(M)$ induced by the preorder ``$\preccurlyeq$",  which coincides with that defined in the Introduction.

For each immersed formal submanifold $(S,(i^{-1}\CO)/\CI)$ of $(M,\CO)$, recall the morphism  \[\iota=(\overline\iota,\iota^*):\ (S,(i^{-1}\CO)/\CI)\rightarrow (M,\CO)\] defined in \eqref{eq:mapiota}.

\begin{prpd}\label{prop:dessub} The morphism \eqref{eq:mapiota} is an injective immersion. Conversely, 
each equivalence class of $\mathrm{Imm}(M)$ has a unique representative of the form $(S,\iota)$, where $S$ is an immersed formal submanifold of $M$ and $\iota$ is as in \eqref{eq:mapiota}.
\end{prpd}
\begin{proof} This is easy by Theorem \ref{thm:charclosedsub} (a). 
\end{proof}

\begin{dfn}
An embedded formal submanifold $S$ of $M$ is called a closed formal submanifold if $S$ is closed in $M$.
\end{dfn}
In the category of smooth manifolds, we also define closed smooth submanifolds of a smooth manifold similarly.

By applying Theorem \ref{thm:charclosedsub} (b), we have the following result. 
\begin{prpd}
There is a canonical bijection between 
\begin{eqnarray}\label{eq:closed}
  \{\text{closed formal submanifold of $M$}\}
\end{eqnarray}  and 
\begin{equation}\label{eq:ideal}
    \{\text{ideal $I$ of $\CO(M)$}\mid \text{$\CO(M)/I$ is a formal $\BC$-algebra}\}.
\end{equation}
Here $\CO(M)/I$ is equipped with the quotient topology.
\end{prpd}

\begin{proof} Let $(S,(i^{-1}\CO)/\mathcal{I})$ be a closed formal manifold. 
   It follows from Theorem \ref{thm:charclosedsub} (b) and 
   \cite[Corollary 4.16]{CSW1} that the continuous homomorphism 
   \be\label{eq:iotaglobal} \iota^*:\ \CO(M)\rightarrow ((i^{-1}\CO)/\CI)(S)\ee 
   is open and surjective. 
   This 
   implies that $\ker \iota^*$ lies in \eqref{eq:ideal}. Then 
   \be \label{eq:MtoI}(S,(i^{-1}\CO)/\mathcal{I}) \mapsto \ker \iota^* \ee
   defines a map from \eqref{eq:closed} to \eqref{eq:ideal}. 
   
   Conversely,  let $I$ be an element in \eqref{eq:ideal}. Then there is a morphism 
  \[
  \mathrm{Specf}(\pi): \ (\mathrm{Specf}(\CO(M)/I),
  \CO_{\CO(M)/I})\rightarrow (\mathrm{Specf}(\CO(M)),\CO_{\CO(M)})=(M,\CO)
  \] of formal manifolds  induced by the quotient map $\pi: \CO(M)\rightarrow \CO(M)/I$ through the functor  \eqref{eq:functor2}. Note that the morphism $\mathrm{Specf}(\pi)$ is a closed embedded immersion by Theorem \ref{thm:charclosedsub} (b).
Proposition \ref{prop:dessub} implies that there is a unique closed formal submanifold $M_I$ such that $(M_I,\iota)$ is equivalent to $((\mathrm{Specf}(\CO(M)/I), \mathrm{Specf}(\pi))$. Note that the image of $M_I$ under the map \eqref{eq:MtoI} is $I$. Then 
   \be\label{eq:ItoM} I \mapsto M_{I} \ee 
   defines a map from  \eqref{eq:ideal}  to \eqref{eq:closed}. 
  
  It is straightforward to see that  \eqref{eq:MtoI} and \eqref{eq:ItoM} are inverse of each other, and then we complete the proof. 
\end{proof}

The following are some examples of closed formal submanifolds. 
\begin{exampled}\label{ex:subunderlineM}
The reduction $\underline{M}$ of $M$ is a closed formal submanifold of $M$ with defining ideal $\m_\CO$. 
\end{exampled} 
\begin{exampled}\label{exam:Ma}
For every $a\in M$, the morphism $M_a\rightarrow M$ produces a closed formal submanifold of $M$, which we still denote as $M_a$. 
By \cite[Proposition 2.6]{CSW1}, the defining ideal of $M_a$ is $\cap_{r\in\BN}\m^r_a$.
\end{exampled}
\begin{exampled}\label{exam:subN'}
    With the notation as in Definition \ref{def:defzetar},  the  $(n',r,k')$-slice 
    $((N')^{(k')},\zeta_{\bm r})$ of $N^{(k)}$ yields a closed formal submanifold of $N^{(k)}$, to be denoted by 
    \[
    (N',i^{-1}\CO_N^{(k)}/\CI_{k',r}).
    \]
\end{exampled}

\subsection{Proof of Theorem \ref{thm:universalformalsub}}\label{subsec:proofthmuniformalsub}
This subsection is aimed to prove Theorem  \ref{thm:universalformalsub}.

Let $\C_M$ denote the sheaf  of  locally constant $\BC$-valued functions on $M$, and write 
 $\Hom_{\BC_M}(\CO,\underline{\CO})$ for the space of all $\BC_M$-homomorphisms between $\CO$ and $\underline{\CO}$.
 Recall that an element $D$ in $\Hom_{\BC_M}(\CO,\underline{\CO})$ is defined to be a family 
\[
\{
  (D_U: \CO(U)\rightarrow \underline{\CO}(U))\}_{U\textrm{ is an open subset of }M}
\]
of linear maps that commute with the restriction maps.  Let $\mathrm{Diff}(\CO,\underline{\CO})$ denote the subspace of $\Hom_{\BC_M}(\CO,\underline{\CO})$ consisting of all differential operators between $\CO$-modules 
$\CO$ and $\underline{\CO}$. See \cite[Section 3.2]{CSW1} for details.

Let $(P,(i^{-1}\underline\CO)/\CI)$ be an immersed smooth submanifold of $\underline M$ as in Theorem  \ref{thm:universalformalsub}. 
Let $V$ be an open subset of $P$, and let $U$ be an open subset of $M$ such that $V\subset U$. In what follows, we will not distinguish the sheaf $(i^{-1}\CO)|_V$  and the pullback $(i|_{V})^{-1}(\CO|_U)$ of $\CO|_U$ through the inclusion $i|_{V}:V\rightarrow U$.
 For every $D\in \Hom_{\BC_U}(\CO|_U,\underline{\CO}|_U)$, write \[i^{-1}D\in \mathrm{Hom}_{\C_V}((i^{-1}\CO)|_V,(i^{-1}\underline{\CO})|_V\] for the $\BC_V$-homomorphism induced by $D$.

For every open subset $V$ of $P$,  
write $\wt{\CI}(V)$ for the subspace 
\be\label{eq:JV}
 \left\{ g\in (i^{-1}\CO)(V): \begin{array}{l}
   (i^{-1}D)_V(g)\in \CI(V) \text{ for every open subset $U$}\\\text{of $M$ containing $V$ and } D\in \mathrm{Diff}(\CO|_U,\underline\CO|_U) 
  \end{array}\right\}
\ee
of $(i^{-1}\CO)(V)$. 
Then the assignment \[\wt\CI: V \mapsto \wt\CI(V)\] forms a subpresheaf of $i^{-1}\CO$ over $P$.

\begin{lemd}\label{lem:Jideal}
    The subpresheaf $\wt\CI$ is  an ideal of  $i^{-1}\CO$ .
\end{lemd}
\begin{proof} Let $V\subset P$ and $U\subset M$ be open subsets such that $V\subset U$. 
Take an open cover  $\{V_\gamma\}_{\gamma\in \Gamma}$ of $V$, and let \[\{g_{\gamma}\in \wt\CI(V_{\gamma})\}_{\gamma\in\Gamma}\] be a family of sections such that $g_{\gamma_1}|_{V_{\gamma_1}\cap V_{\gamma_2}}=g_{\gamma_2}|_{V_{\gamma_1}\cap V_{\gamma_2}}$ for all $\gamma_1,\gamma_2\in \Gamma$. Write  $g$ for the section 
 in $(i^{-1}\CO)(V)$ satisfying $g|_{V_{\gamma}}=g_{\gamma}$ for all $\gamma\in \Gamma$. 
Then for every $D\in\mathrm{Diff}(\CO|_U,\CO|_U)$ we have that 
\[((i^{-1}D)_V(g))|_{V_{\gamma}}=(i^{-1}D)_{V_{\gamma}}(g_{\gamma})\in \CI(V_{\gamma}).\]
 This implies that $(i^{-1}D)_V(g)\in \CI(V)$, and so $\wt\CI$ is a sheaf over $P$.

Let $h\in (i^{-1}\CO)(V)$, and  
take an open cover $\{V'_{\gamma}\}_{\gamma\in \Gamma'}$ of $V$ such that \[ h|_{V_{\gamma}'}\in (i_{p}\CO)(V'_{\gamma})=\varinjlim_{V'_\gamma\subset U'}\CO(U'),\] where $i_p\CO$ is the presheaf obtained by the pullback of $\CO$. 
Let $h|_{V'_\gamma}$ be represented by  a section $f_{\gamma}\in \CO(W_\gamma)$ on an open neighborhood  $W_\gamma$ of $V'_\gamma$ in $M$. 
Then for every   $g\in \wt\CI(V)$, we have that 
\[((i^{-1}D)_V(hg))|_{V'_{\gamma}}=(i^{-1}D_\gamma)_{V'_\gamma}(g|_{V'_\gamma})\in \CI(V'_{\gamma}),\]
where $D\in\mathrm{Diff}(\CO|_U,\CO|_U)$, and (see  \cite[(3.5)]{CSW1})
 \[D_\gamma:=D|_{W_{\gamma}\cap U}\circ f_{\gamma}|_{W_{\gamma}\cap U}\in \mathrm{Diff}(\CO|_{W_{\gamma}\cap U},\CO|_{W_{\gamma}\cap U}).\]
 This, together with the sheaf property of $\CI$, gives that $hg\in \wt\CI(V)$, as required. 

\end{proof}

 Let $n,n',k,r\in \BN$ with $n-n'=r$,  $N$  an open subset of $\BR^n$ containing $0$, and $N'$ an $n'$-slice of $N$ as in \eqref{eq:N'}.
\begin{lemd}\label{lem:prethmunisub} Assume that $M=N^{(k)}$ and $P=(N')^{(0)}$. Then $\wt\CI=\CI_{k+r,r}$ (see Example \ref{exam:subN'}). 
\end{lemd}
\begin{proof}
Let $V$ be an open subset of $N'$ and  $g\in (i_p\CO)(V)$, where $i_p\CO$ denotes the presheaf obtained by the pullback of $\CO$.
Recall from \cite[Example 3.11]{CSW1} that
every element in $\mathrm{Diff}(\CO,\underline{\CO})$ has the form 
\[
\sum_{I\in \BN^n, J\in \BN^k} f_{I,J}\circ \partial_x^I\partial_y^J,
\]
where $\{f_{I,J}\}_{I\in \BN^n, J\in \BN^k}$ is a locally finite family of smooth functions on $N$. This, together with Example \ref{ex:closedsub}, implies that the following three conditions are equivalent to each other.
\begin{itemize}
    \item [(a)] $g\in \wt I(V)$.
    \item [(b)] For each formal function $f \in \CO(W)$ defined on an open neighborhood $W$ of $V$ in $N$ which represents $g$, we have that  $X(f)\circ i=0$ for each $X\in \mathrm{Diff}(\CO(W),\underline{\CO}(W))$.
    \item [(c)] $g\in \CI_{k+r,r}(V)$.
\end{itemize}
The lemma then follows from the fact that $\wt\CI$ and $\CI_{k+r,r}$ are both subsheaves of $i^{-1}\CO$.

\end{proof}

\vspace{3mm}
\noindent\textbf{Proof of Theorem \ref{thm:universalformalsub}:}  
By Lemma \ref{lem:Jideal}, $(P,i^{-1}\CO/\wt{\CI})$ is a locally ringed space over $\Spec(\BC)$, to be denoted by $\wt{P}$. 
For every $a\in P$, pick a chart $(U_a,(N_a)^{(k_a)},\vartheta_a=(\overline{\vartheta_a},\vartheta_a^*))$ of $M$ containing $a$, and an open neighborhood $V_a$ of $a$ in $P$ such that $N_a':=\overline{\vartheta_a}(V_a)$ is a $n_a'$-slice in $N_a$, where $n_a':=\dim_a P$.
By Lemma \ref{lem:prethmunisub}, we have that $(V_a,(i^{-1}\CO/\wt{\CI})|_{V_a})$ is a formal manifold, and so $\wt{P}$ is an immersed formal submanifold of $M$. 

Let $\psi=(\overline{\psi},\psi^*): M' \rightarrow M$ be a morphism of formal manifolds. Assume that 
$\overline{\psi}(M')\subset P$ and  the induced map $\overline{\psi}: M'\rightarrow P$  is continuous.  
For every $a\in P$,  
write
 \[
 \psi_{U_a}=(\overline{\psi_{U_a}}, \psi_{U_a}^*):\ (\overline{\psi}^{-1}(U_a),\CO'|_{{\overline\psi}^{-1}(U_a)})\rightarrow (U_a, \CO|_{U_a})\]
for the morphism of formal manifolds induced by $\psi$. 
It follows from Lemmas \ref{lem:prethmunisub} and \ref{lem:constructphi} that there is a unique morphism 
\[
\phi_{U_a}=(\overline{\phi_{U_a}},\phi_{U_a}^*):\ 
(\overline{\psi}^{-1}(U_a),\CO'|_{{\overline\psi}^{-1}(U_a)})\rightarrow (V_a,(i^{-1}\CO/\wt\CI)|_{V_a})
\]
such  that the diagram
\[\xymatrix{ & (V_a,(i^{-1}\CO/\wt{
\CI
})|_{V_a}) \ar[d]^{\iota}\\
  \quad 
(\overline{\psi}^{-1}(U_a),\CO'|_{{\overline\psi}^{-1}(U_a)}) \ar[r]^{\psi_{U_a}}\ar[ur]^{\phi_{U_a}}& (U_a, \CO|_{U_a})}
  \]
   commutes.
The uniqueness allows us to glue $\{\phi_{U_a}\}_{a\in M}$ together, and we obtain in this way a morphism 
$\phi: M'\rightarrow \wt P$ satisfying $\iota\circ \phi=\psi: M'\rightarrow M$. By  Proposition \ref{prop:sub=mono}, such a morphism $\phi$ is unique. Thus $\wt P$ has the universal property as stated in the Theorem \ref{thm:universalformalsub}. On the other hand, the uniqueness of such an immersed formal submanifold $\wt P$ follows from Proposition \ref{prop:dessub}.  This finishes the proof. 
\qed

\subsection{Applications of Theorem \ref{thm:universalformalsub}} \label{subsection:app}
Here we give some by-products of  Theorem \ref{thm:universalformalsub}. Recall the notion of initial immersed formal (or smooth) submanifolds from the Introduction. In this subsection,  let $P$ be an initial immersed smooth submanifold of $\underline{M}$.
We have the following obvious result by Theorem \ref{thm:universalformalsub}. 
 \begin{cord}\label{cor:universalformalsub}
There is a unique initial immersed formal submanifold $\wt P$ of $M$ such that $\wt P=P$ as topological spaces.
\end{cord}


Note that the set \be \label{eq:submanifoldoverP}\{\text{immersed formal submanifold $P'$ of $M$}\mid \text{$P'=P$ as subsets of $M$}\}\ee admits a partial order induced by \eqref{eq:preorder}.
 Corollary \ref{cor:universalformalsub} implies the following result, which
 says that  
 $P$ is the minimal one
and $\wt P$ is the maximal one in the partially ordered set \eqref{eq:submanifoldoverP}.

\begin{cord} Let $P'$ be an immersed formal submanifold of $M$ such that $P'=P$ as subsets of $M$. Then there is a unique morphism $P\rightarrow P'$ as well as a unique morphism $P'\rightarrow \wt P$ of formal manifolds such that the diagram 
\[\xymatrix{ P\ar[r]\ar[rd]_{\iota}&P'\ar[r]\ar[d]^{\iota}&\wt P\ar[ld]^{\iota}\\
&M}
  \] 
commutes. 
\end{cord}

By Proposition \ref{prop:sub=mono},   the map
\begin{equation}\begin{split} \label{eq:morMatoMinj}
     \{\text{morphism from $M'$ to $\wt P$}\}&\rightarrow 
  \{\text{morphism from $M'$ to $M$}\},\\
  \phi&\mapsto 
 \iota\circ \phi 
\end{split}\end{equation}
is injective, where  $M'$ is another formal manifold. 
In view of Corollary \ref{cor:universalformalsub}, we describe the image of \eqref{eq:morMatoMinj} as follows.

\begin{cord} \label{cor:app3}
  Let 
  \[\psi=(\overline\psi,\psi^*):\ (M',\CO')\rightarrow (M,\CO)\] be  a morphism of formal manifolds. Then  $\psi$ lies in the image of \eqref{eq:morMatoMinj} if and only if 
   $\overline\psi(M')\subset P$.
 \end{cord}
   
Note that, if $P=\{a\}$ for some $a\in M$, then $\wt P=M_a$ (see Example \ref{exam:Ma}). 
As a special case of Corollary \ref{cor:app3}, we have the following result. 
\begin{cord}
For every morphism $\psi=(\overline\psi,\psi^*): M'\rightarrow M$ satisfying that $\overline{\psi}(b)=a$ for all $b\in M'$, it induces a morphism from $M'$ to $M_a$. 
   Conversely, every morphism from $M'$ to $M_a$ arises in this way. 
\end{cord}
\subsection{Proof of Theorem \ref{thm:levelsets}}\label{subsec:levelsets}
In this subsection, let \[\varphi=(\overline\varphi, \varphi^*):\ (M',\CO')\rightarrow (M,\CO)\] be a morphism of formal manifolds, and let $a\in M$. 
The main goal is to construct a locally ringed space $\varphi^{-1}(a)$.  And then to prove that it is a formal manifold and satisfies the Cartesian diagram in Theorem \ref{thm:levelsets},  provided that  $\varphi$ is standardizable near $b$ for every $b\in \overline{\varphi}^{-1}(a)$.



For this purpose, we first provide a generalization of formal spectra (see \cite[Section 5.3]{CSW1}) in the setting of topological $\BC$-algebras. Let $A$ be a topological $\BC$-algebra.  Recall that a continuous character of $A$ is a continuous $\BC$-algebra homomorphism $A\rightarrow \BC$. Define the formal spectrum of $A$ to be the set
\be\label{eq:specf}
  \mathrm{Specf}(A):=\{\textrm{continuous character of $A$}\},
\ee and equip it with the coarsest topology such that the map
\[
  \mathrm{Specf}(A)\rightarrow \BC, \quad \chi\mapsto \chi(g)
\]
is continuous for all $g\in A$.   For every $\chi\in \mathrm{Specf}(A)$, write $A_\chi$ for the localization of the ring $A$ at the maximal ideal $\ker\chi$. 
Define
\be \label{eq:OA}
  \CO_A:\ U\mapsto \CO_A(U):=\left\{g\in \prod_{\chi\in U} A_\chi\mid g\textrm{ is locally represented by $A$ }\right\}.
\ee Here $U$ is an open subset of $\mathrm{Specf}(A)$, $g$ is called  locally represented by $A$ if for every $\chi_0\in U$,
 there is an open neighborhood $U_0$ of $\chi_0$ in $U$,
 and an element $g_0\in A$ such that  $\{g_\chi\}_{\chi\in U_0}$ equals the image of $g_0$ under the canonical map
\[ A\rightarrow \prod_{\chi\in U_0}A_\chi.\]
Together with the obvious restriction maps, $\CO_A$ is a sheaf of $\BC$-algebras over $\mathrm{Specf}(A)$. Then $(\mathrm{Specf}(A),\CO_{A})$ is a good locally ringed space over $\Spec(\BC)$. Here and hence follow,  we say that a locally ringed space $(M_1, \CO_1)$ over $\Spec (\BC)$ is good, if for each point $b_1\in M_1$\[\CO_{1,b_1}/\m_{b_1}=\BC \quad (\text{$\m_{b_1}$ is the maximal ideal of the stalk $\CO_{1,b_1}$}).\]  
All good locally ringed spaces over $\Spec (\BC)$ form a full subcategory of the category of locally ringed spaces over $\Spec (\BC)$.

Let $\phi: A'\rightarrow A$ be a continuous homomorphism of topological $\BC$-algebras. Then it induces a continuous map 
\be\label{eq:overlinephi} \overline{\phi}:\  \mathrm{Specf} (A)\rightarrow \mathrm {Specf} (A'),\quad \chi\mapsto \chi \circ \phi\ee and a homomorphism 
\[\phi^*_{\chi}:\ A'_{\overline{\phi}(\chi)}\rightarrow A_{\chi}\]
for every $\chi\in \mathrm{Specf}(A)$.
 Similar to \cite[Lemma 5.10]{CSW1}, there is a morphism 
    \[\mathrm{Specf}(\phi)=(\overline{\phi}, \phi^*):\  (\mathrm{Specf} (A),\CO_A)\rightarrow (\mathrm {Specf} (A'), \CO_{A'})\] of good locally ringed spaces over $\Spec (\BC)$  such that for every open subset $U'\subset \mathrm{A'}$, the diagram 
\[\begin{CD}
\CO_{A'}(U')@>\phi^*_{U'}>> \CO_A(U)\\ @VVV @VVV \\ A'_{\overline{\phi}(\chi)}@>\phi^*_{\chi}>> A_{\chi}
\end{CD} \quad \text{($U:=\overline{\phi}^{-1}(U')\subset \mathrm{Specf}(A)$)}\]
    commutes for all $\chi\in U$.

 Then we have a contravariant functor
\be\label{eq:functortospace}
   A\mapsto (\mathrm{Specf}(A),\CO_A) \qaq \phi\mapsto \mathrm{Specf}(\phi)
  \ee
 from the category of topological  $\BC$-algebras to the category of good locally ringed spaces over $\Spec (\BC)$.

As usual, if  $\mathrm{Specf}(A)=\emptyset$, then $\CO_A$ is defined by $\emptyset \mapsto \{0\}$. When $\mathrm{Specf}(A)$ is nonempty,
we view $A$ as a subalgebra of $\CO_A(\mathrm{Specf}(A))$ through the canonical injection between them.  

\begin{lemd}\label{lem:thestrangeform}
    Let $(M_1,\CO_1)$ be a good locally ringed space over $\Spec (\BC)$,  \[\overline{\phi}:\ M_1\rightarrow \mathrm{Specf}(A)\]  a continuous map between topological spaces, and  \[\phi: \ A\rightarrow \CO_1(M_1)\]  a homomorphism of $\BC$-algebras. 
    Assume that $\mathrm{Specf}(A)$ is nonempty, and 
    \be \label{eq:condition}\mathrm{Ev}_{b_1}(\phi(g))=\overline{\phi}(b_1)(g)\ee for each $b_1\in M_1$ and $g\in A$,
    where $\mathrm{Ev}_{b_1}$ is the composition map  \[ \CO_1(M_1)\rightarrow \CO_{1,b_1}\rightarrow \CO_{1,b_1}/\m_{b_1}=\BC.\] Then there is a unique  sheaf homomorphism \[\phi^*: \ \overline{\phi}^{-1}\CO_A\rightarrow \CO_1\] such that   the restriction \[\phi^*|_{A}:\  A\rightarrow \CO_1(M_1)\]  of $\phi^*: \CO_A(\mathrm{Specf}(A))\rightarrow \CO_1(M_1)$ equals $\phi$. 
\end{lemd}
\begin{proof}
    By \eqref{eq:condition}, 
    there is  a unique local homomorphism $\phi^*_{b_1}: A_{\overline{\phi}(b_1)} \rightarrow \CO_{1,b_1}$
    such that the diagram 
    \[\begin{CD}A@>\phi>> \CO_1(M_1)\\
    @VVV @VVV\\
    A_{\overline{\phi}(b_1)} @>\phi^*_{b_1}>> \CO_{1,b_1}
    \end{CD}\]commutes.
    As shown in the proof of \cite[Lemma 5.10]{CSW1}, for every open subset $U$ of $\mathrm{Specf
    }(A)$ and  every open subset $U_1$ of $M_1$ such that $\overline{\phi}(U_1)\subset U$, there is  a $\BC$-algebra homomorphism 
    \[\phi^*_{U,U_1}:\ \CO_{A}(U)\rightarrow \CO_1(U_1)\] such that the diagram 
    \[ \begin{CD}
               \CO_{A}(U) @> \phi^*_{U,U_1}  >>  \CO_{1}(U_1)\\
            @V  VV           @V V V\\
            A_{\overline{\phi}(b_1)} @> \phi_{b_1}^*>>  \CO_{1,b_1} \\
  \end{CD}\]commutes. Note that the family $$\{\phi^*_{U,U_1}: \CO_{A}(U)\rightarrow \CO_1(U_1)\}_{\,\textrm{ $U$ and $U_1$ are as above 
  }}$$ is compatible with respect to the restriction maps. This gives  a sheaf homomorphism
 \[
   \phi^*: \ \overline \phi^{-1} \CO_A\rightarrow \CO_1
 \] satisfying  \[\phi^*|_{A}=\phi:\  A\rightarrow \CO_1(M_1).\]

 The uniqueness of such a sheaf homomorphism follows from the fact that the canonical map $\CO_1(U_1)\rightarrow \prod_{b_1\in U_1}\CO_{1,b_1}$ is injective. 
\end{proof}

Let $I'$ denote the ideal of $\CO'(M')$  generated by the image of \[\{f\in \CO(M)\,:\, f(a)=0\}\] under the map $\varphi^*:\CO(M)\rightarrow \CO'(M')$. 
Equip $\CO'(M')/I'$ with the quotient topology, which makes it  a topological $\BC$-algebra. 
 By  the functor \eqref{eq:functortospace} and \cite[Theorem 5.7]{CSW1}, there is   a morphism
\[\mathrm{Specf}(\pi)=(\overline{\pi},\pi^*):\ \mathrm{Specf}(\CO'(M')/I')\rightarrow \mathrm{Specf}(\CO'(M'))=M' \] of good locally ringed spaces over $\Spec (\BC)$ induced by the quotient homomorphism
\[\pi:\ \CO'(M')\rightarrow \CO'(M')/I'.\]
Here and below, if there is no confusion, we will often not distinguish a locally ringed space $(M_1,\CO_1)$ with its underlying topological space $M_1$.

Note that 
the continuous map \be\label{eq:overlinepi}  \overline{\pi}:\ \mathrm{Specf}(\CO'(M')/I') \rightarrow \mathrm{Specf}(\CO'(M'))=M',\quad \chi\mapsto \chi \circ \pi\ee is injective.
 
\begin{lemd}\label{lem:=-1a} The image of \eqref{eq:overlinepi} is  $\overline{\varphi}^{-1}(a)$. Furthermore, the induced map
\be \label{eq:resofpi} \overline{\pi}:\ \mathrm{Specf}(\CO'(M')/I')\rightarrow \overline{\varphi}^{-1}(a)\ee of \eqref{eq:overlinepi} is  a topological isomorphism, where $\overline{\varphi}^{-1}(a)$ is endowed with the subspace topology of $M'$.
\end{lemd}
\begin{proof} 
Let $\chi\in \mathrm{Specf}(\CO'(M')/I')$.  Then for each $f\in \CO(M)$ satisfying $f(a)=0$, we have that \[f(\overline{\varphi}\circ \overline{\pi}(\chi))=\varphi^*f(\overline{\pi}(\chi))= \chi \circ \pi (\varphi^*f)=0.\] This implies that $\overline\varphi \circ \overline{\pi}(\chi) =a$ and then $\overline{\pi}(\chi)\in \overline{\varphi}^{-1}(a)$.

On the other hand, for every $b\in \overline{\varphi}^{-1}(a)$,
recall the continuous character \[\mathrm{Ev}_b:\ \CO'(M')\rightarrow \CO'_{b}\rightarrow \CO'_{b}/\m_b=\BC.\]
Since $\mathrm{Ev}_b(I')=0$, there is  a character $\chi_b: \CO'(M')/I'\rightarrow \BC$ such that $\chi_b\circ \pi=\mathrm{Ev}_b$.
The first assertion then follows. 

For the second one, 
note that the topology on $\mathrm{Specf}(\CO'(M')/I')$ 
coincides with the coarsest topology on $\mathrm{Specf}(\CO'(M')/I')$ such that the injective map
\[\mathrm{Specf}(\CO'(M')/I')\rightarrow \prod_{g\in \CO'(M')/I'}\BC, \quad \chi \mapsto \{\chi(g)\}_{g\in \CO'(M')/I'}\] is continuous. 
Then by \cite[Lemma 5.8]{CSW1}, we have that the map \eqref{eq:resofpi} is a topological isomorphism.
\end{proof}


The following lemma is straightforward (\cf the proof of Proposition \ref{prop:sub=mono}).
\begin{lemd}\label{lem:monoofringedspace2}
     Let  $\phi=(\overline{\phi},\phi^*): (M_1,\CO_1)\rightarrow (M_2,\CO_2)$ be a morphism of good locally ringed spaces over $\Spec(\C)$. Assume that $\overline{\phi}$ is injective and that for each $b_1\in M_1$, the local homomorphism  \[\CO_{2,\overline{\phi}(b_1)}\rightarrow \CO_{1,b_1}\] is surjective. Then $\phi$ is a monomorphism in the category of good locally ringed spaces over $\Spec(\BC)$.
\end{lemd}

Let $U$ be an open neighborhood of $a$ in $M$, and let $U'$ be an open subset of $M'$ such that $\overline{\varphi}(U')\subset U$. 

\begin{lemd}\label{lem:monoofringedspace}
      The  morphism 
      \[
      \mathrm{Specf}(\pi)|_{\overline{\pi}^{-1}(U')}:\  ( \overline{\pi}^{-1}(U'), \CO_{\CO'(M')/I'}|_{\overline{\pi}^{-1}(U')})\rightarrow (U',\CO'|_{U'})\] is a monomorphism in the category of good locally ringed spaces over $\Spec (\BC)$.
\end{lemd}
\begin{proof}
 The assertion follows from  Lemma \ref{lem:monoofringedspace2} and the fact that for each $\chi\in \mathrm{Specf}(\CO'(M')/I')$, the local homomorphism \[\CO'_{\overline{\pi}(\chi)}\rightarrow (\CO'(M')/I')_{\chi}\] induced by $\pi$ is surjective. 

\end{proof}

Recall the morphism $\varsigma_a$ from \eqref{eq:varsigma}.
\begin{lemd}\label{lem:comofSOJ}
The diagram 
 \be\label{eq:localcart} \begin{CD}
    ( \overline{\pi}^{-1}(U'), \CO_{\CO'(M')/I'}|_{\overline{\pi}^{-1}(U')})@>>> \Spec(\BC)\\ @V \mathrm{Specf}(\pi)|_{\overline{\pi}^{-1}(U')}VV @VV\varsigma_a V \\ (U',\CO'|_{U'})@>\varphi|_{U'}>>(U,\CO|_{U}) 
 \end{CD}\ee 
 commutes.
    
\end{lemd}
\begin{proof}
It suffices to prove that the assertion holds when $U=M$ and $U'=M'$. In this case, 
       through the functor \eqref{eq:functortospace}, the lemma follows from the obvious commutative diagram 
   \[\begin{CD}
        \CO(M)@>\varphi^*>>\CO'(M')\\@V\mathrm{Ev}_aVV @VV\pi V \\\BC@>>> \CO'(M')/I'.
    \end{CD}\]
\end{proof}

\begin{lemd}\label{lem:pullback}
   The commutative diagram \eqref{eq:localcart}  is a Cartesian diagram in the category of good locally ringed spaces over $\Spec(\BC)$.
\end{lemd}
\begin{proof}Let $(M_1,\CO_1)$ be a good locally ringed space over $\mathrm{Spec}(\BC)$, and 
    \[\psi=(\overline{\psi},\psi^*):\ (M_1,\CO_1)\rightarrow (U',\CO'|_{U'})\]
    a morphism of good locally ringed spaces over $\Spec(\BC)$ such that the diagram 
    \be\label{eq:diapullback} \begin{CD}
        M_1 @>>> \Spec (\BC) \\ @V\psi VV @ VV\varsigma_a  V\\
        U' @>\varphi|_{U'}>>U
    \end{CD}\ee commutes. We aim to prove there is a unique morphism
    \[\theta: \ (M_1,\CO_1)\rightarrow ( \overline{\pi}^{-1}(U'), \CO_{\CO'(M')/I'}|_{\overline{\pi}^{-1}(U')})\] such that the diagram 
 \be\label{eq:univeraldiagg} \xymatrix{&&&\mathrm{Spec}(\BC)\\M_1\ar[urrr]\ar[drrr]^{\psi}\ar[rrr]^{\theta}&&&\overline{\pi}^{-1}(U')\ar[d]^{\mathrm{Specf}(\pi)|_{\overline{\pi}^{-1}(U')}}\ar[u]\\&&& U'}\ee
    commutes.
    The uniqueness of such a morphism $\theta$ follows from the fact that  $\mathrm{Specf}(\pi)|_{\overline{\pi}^{-1}(U')}$ is a monomorphism (see Lemma \ref{lem:monoofringedspace}). In what follows, we prove the existence of $\theta$.


  We first consider the special case when $U=M$ and $U'=M'$. By \eqref{eq:diapullback}, we have that $\mathrm{im}\, \overline{\psi}\subset\overline{\varphi}^{-1}(a)$. In view of Lemma \ref{lem:=-1a},
   there is a continuous  map \be\label{eq:barphi} \overline{\phi}:\ M_1\rightarrow \mathrm{Specf}(\CO'(M')/I')\ee induced by  $\overline{\psi}$. 
On the other hand, for each $f\in \CO(M)$ with $f(a)=0$, the image of $f$ under the composition map
    \[\varphi^*\circ \psi^*:\ \CO(M)\rightarrow \CO'(M')\rightarrow \CO_1(M_1)\] is $0$ by \eqref{eq:diapullback}. Thus $\psi^*(I')=0$. Then there is a unique  continuous homomorphism \[\phi: \CO'(M')/I'\rightarrow \CO_1(M_1)\] of topological $\BC$-algebras  such that the diagram 
    \[\xymatrix{\CO'(M')\ar[d]_{\pi}\ar[dr]^{\psi^*}\\\CO'(M')/I'\ar[r]^{\phi}&\CO_1(M_1)}\]
   commutes. Note that for each $b_1\in M_1$ and $g\in \CO'(M')/I'$, 
   \[\mathrm{Ev}_{\overline{\phi}(b_1)}(g)=\mathrm{Ev}_{\overline\psi(b_1)}(\wt{g})=\mathrm{Ev}_{b_1}(\psi^*(\wt g))=\mathrm{Ev}_{b_1}(\phi(g)),\] 
where $\wt g\in \pi^{-1}(g)$. Then by Lemma \ref{lem:thestrangeform},  we get  a 
 unique morphism 
    \[\mathrm{Specf}(\phi):=(\overline{\phi}, \phi^*):\  M_1\rightarrow \mathrm {Specf} (\CO'(M')/I')\] of locally ringed spaces   such that \[\phi^*|_{\CO'(M'/I')}=\phi:\ \CO'(M’)/I'\rightarrow \CO_1(M_1).\]
    
    Let $b_1\in M_1$. Note that a local homomorphism \[\CO'_{\overline\psi(b_1)}\rightarrow \CO_{1,b_1} \] which satisfies the commutative diagram 
    \[\begin{CD}\CO'(M')@>\psi^*>> \CO_1(M_1)\\
    @VVV @VVV\\
    \CO'_{\overline{\psi}(b_1)} @>>> \CO_{1,b_1}
    \end{CD}\] is unique. This forces that $\psi^*_{b_1}=\phi^*_{b_1}\circ \pi^*_{\overline{\phi}(b_1)}$ and hence $\psi=\mathrm{Specf}(\pi)\circ \mathrm{Specf}(\phi)$.  Therefore, the diagram \eqref{eq:univeraldiagg} 
   commutes by taking $\theta=\mathrm{Specf}(\phi)$.

    For the general case, 
    note that the image of $\overline{\phi}$ 
    is contained in $\overline{\pi}^{-1}(U')$. Here $\overline{\phi}: M_1\rightarrow \mathrm{Specf}(\CO'(M)/I')$ is the map induced by the composition map 
    \[ M_1\xrightarrow{\psi} U'\rightarrow M'\] as in \eqref{eq:barphi}.
    Then  $\mathrm{Specf}(\phi)$ induces a morphism 
   \[\theta:\ (M_1,\CO_1)\rightarrow (\overline{\pi}^{-1}(U'),\CO_{\CO'(M')/I'}|_{\overline{\pi}^{-1}(U')})\] between good locally ringed spaces over $\Spec(\BC)$, which makes
    the diagram \eqref{eq:univeraldiagg}  commutes.  This finishes the proof.

\end{proof}


   Let $n,k,n',k'\in \BN$, and set   (see \eqref{eq:C(n)}) \[N'=\mathrm{C}(n')\subset \BR^{n'},\qquad  N=\mathrm{C}(n)\subset \BR^{n}. \]
   Write $(x_1,x_2,\dots,x_n,y_1,y_2,\dots,y_k)$ and $(u_1,u_2,\dots,u_{n'},z_1,z_2,\dots,z_{k'})$ for the standard coordinate systems of $N^{(k)}$ and $(N')^{(k')}$, respectively.
\begin{lemd}\label{lem:levlinstand}
  Assume that $M'=(N')^{(k')}$, $M=N^{(k)}$, $a=0$, and that the morphism $\varphi$ is given by  \begin{equation}\begin{split}\label{eq:introstrandard1}&(x_1,\dots,x_{r_1},x_{r_1+1},\dots,x_{n-r_2},x_{n-r_2+1},\dots,x_n,y_1,\dots,y_{r_3},y_{r_3+1},\dots,y_k)\\\mapsto\ & (u_1,\dots,u_{r_1},\underbrace{0,\dots,0}_{n-r_1-r_2},z_1,\dots,z_{r_2},z_{r_2+1},\dots,z_{r_2+r_3},\underbrace{0,\dots,0}_{k-r_3})\end{split}\end{equation}for some $r_1,r_2,r_3\in \BN$. Then $\mathrm{Specf}(\CO'(M')/I')$ is a formal manifold. 
\end{lemd}
\begin{proof}
    Let $n_1=n'-r_1$, $k_1=k'-r_2-r_3$, and set \[N_1:=\{(b_1,\dots,b_{n'})\in N'\mid b_{i}=0 \quad \text{for every $i\leq r_1$}\}.\] Denote by $(u'_1,u'_2,\dots,u'_{n_1},z_1',z_2',\dots,z'_{k_1})$ the standard coordinate system of $N_1^{(k_1)}$.
    Then we have the following commutative  diagram  \be \label{eq:localca}\begin{CD}(N_1)^{k_1} @>>>\Spec (\BC)\\ @V\zeta VV @VV\varsigma_a V \\ (N')^{(k')}@>\varphi>> N^{(k)}, \end{CD}\ee
 where $\zeta=(\overline\zeta,\zeta^*)$ is given by    
   \begin{equation}\begin{split}\label{eq:introstrandard1}&(u_1,\dots,u_{r_1},u_{r_1+1},\dots,u_{n'},z_1,\dots,z_{r_2+r_3},z_{r_2+r_3+1},\dots,z_{k'})\\\mapsto\ & (\underbrace{0,\dots,0}_{r_1},u'_1,\dots,u'_{n_1},\underbrace{0,\dots,0}_{r_2+r_3},z'_1,\dots,z'_{k_1}).\end{split}\end{equation} 
 Note that $\overline{\zeta}$ is the natural inclusion  and for every $g= \sum_{J\in \BN^{k'}}g_Jz^J\in\CO_{N'}^{(k')}(V')$, 
   \be \label{eq:zerainproof}\zeta^*_{V'}(g)=\sum_{J\in \BN^{k_1}}(g_{(0,J)}\circ \overline{\zeta})(z')^J\ee where $V'$ is an open subset of  $N'$ and \[(0,J):=(\underbrace{0,\dots,0}_{r_2+r_3},j_1,\dots,j_{k_1})\in \BN^{k'}\] for $J=(j_1,\dots,j_{k_1})\in \BN^{k_1}$.
   Since $\zeta$ is a closed embedded immersion, the continuous homomorphism \[\zeta^*: \  \CO_{N'}^{(k')}(N')\rightarrow \CO_{N_1}^{(k_1)}(N_1)\] is open and surjective (see Theorem \ref{thm:charclosedsub} (b) and \cite[Corollary 4.16]{CSW1}). 
   
We claim that $I'=\ker \zeta^*$. This will imply that  $\CO_{N'}^{(k')}(N')/I'=\CO_1^{(k_1)}(N_1)$ as topological $\BC$-algebras. Then \[\mathrm{Specf}(\CO_{N'}^{(k')}(N')/I')= \mathrm{Specf}(\CO_{N_1}^{(k_1)}(N_1))=N_1^{(k_1)}\] is a formal manifold (see \cite[Theorem 5.7]{CSW1}), as required. 

Now we turn to prove this claim. If $f\in \CO_N^{(k)}(N)$ with $f(0)=0$, then $\varphi^*(f)\in \ker \zeta^*$ by the diagram \eqref{eq:localca}. Therefore $I'\subset \ker \zeta^*$. 
   On the other hand, let $g\in \ker \zeta^*$. 
  By \eqref{eq:zerainproof}, $g$ has a form  
   \[\sum_{J\in \BN^{k_1}}g_{(0,J)}z^{(0,J)}+\sum_{1\leq i\leq r_2+r_3}z_ih_i\]
  for some $h_i\in \CO_{N'}^{(k')}(N')$ and $g_{(0,J)}\in \mathrm{C}^{\infty}(N')$ with $g_{(0,J)}\circ \overline{\zeta}=0$.  For the case that $r_1=0$, we have $g_{(0,J)}=0$  for each $J\in \BN^{k_1}$ and so $g\in I'$.
 For the case that $r_1\geq 1$,  we have $g_{(0,J)}/u_1\in \mathrm{C}^{\infty}(N')$  for each $J\in \BN^{k_1}$ . This implies that \[g=u_1\cdot(\sum_{J\in \BN^{k_1}}\frac{g_{0,J}}{u_1}z^{(0,J)})+\sum_{1\leq i\leq r_2+r_3}z_ih_i\in I'.\]  So we obtain that $I'=\ker \zeta^*$, and then the proof is finished. 

   \end{proof}


\begin{lemd}\label{lem:lastlem}
    Assume that  for each $b\in \overline{\varphi}^{-1}(a)$, the morphism $\varphi$ is standardizable near $b$. Then $\mathrm{Specf}(\CO'(M')/I')$ is a formal manifold. Furthermore, the morphism \[\mathrm{Specf}(\pi): \ \mathrm{Specf}(\CO'(M')/I')\rightarrow M' \] of formal manifolds is a closed embedded immersion.
\end{lemd}
\begin{proof} Lemmas \ref{lem:pullback} and \ref{lem:levlinstand} imply that $\mathrm{Specf}(\CO(M)/I')$ is a formal manifold, while 
Theorem \ref{thm:charclosedsub} implies that  $\mathrm{Specf}(\pi)$ is a closed embedded immersion. 
 \end{proof}

Theorem \ref{thm:levelsets} follows from Lemma \ref{lem:lastlem}
and Lemma \ref{lem:pullback} (with $U=M$ and $U'=M'$ in the diagram \eqref{eq:localcart}). 

\section*{Acknowledgement}
Chen Fulin is supported by the Natural Science Foundation of Xiamen, China (No. 3502Z202473005), the National Natural Science Foundation of China (Nos. 12131018 and 12471029) and the Fundamental Research Funds for the Central Universities (No. 20720230020).
	Sun Binyong is supported by  National Key R \& D Program of China (Nos. 2022YFA1005300 and 2020YFA0712600) and New Cornerstone Investigator Program.
	The first and the third authors would like to thank Institute for Advanced Study in Mathematics, Zhejiang University. Part of this work was carried out while they were visiting the institute.

\end{document}